\documentclass[journal]{IEEEtran}
\usepackage{graphicx}
\usepackage{amsfonts}
\usepackage{amsmath}
\usepackage{cite}
\usepackage{color}
\usepackage{amssymb}

\newtheorem{theorem}{\bf Theorem}

\newtheorem{lemma}{\bf Lemma}
\newtheorem{definition}{\bf Definition}

\newtheorem{assumption}{\bf Assumption}
\newenvironment{proof}{{\it \bf Proof:}}{\hfill $\blacksquare$\par}
\newtheorem{Remark}{\bf Remark}
\begin{document}

\title{Evolutionary dynamics under periodic switching of update rules on regular networks}

\author{Shengxian~Wang,~Weijia~Yao,~Ming~Cao, and~Xiaojie~Chen
\thanks{This research was supported by the National Natural Science Foundation of China (Grant Nos. 61976048 and 62036002) and the Fundamental Research Funds of the Central Universities of China. S.W. acknowledges the support from China Scholarship Council (Grant No. 202006070122). (Corresponding authors:  Ming Cao and Xiaojie Chen).}
\thanks{S. Wang is with School of Mathematical Sciences, University of Electronic Science and Technology of China, Chengdu 611731, China,  and  also with ENTEG, Faculty of Science and Engineering, University of Groningen, Groningen 9747 AG, The Netherlands (e-mail: shengxian.wang.yy@outlook.com).}
\thanks{W. Yao is with School of Robotics, Hunan University, Hunan, China, and also with ENTEG, Faculty of Science and Engineering, University of Groningen, Groningen 9747 AG, The Netherlands (e-mail: weijia.yao.new@outlook.com).}
\thanks{M. Cao is with ENTEG, Faculty of Science and Engineering, University of Groningen, Groningen 9747 AG, The Netherlands (e-mail: m.cao@rug.nl).}
\thanks{X. Chen is with School of Mathematical Sciences, University of Electronic Science and Technology of China, Chengdu 611731, China (e-mail: xiaojiechen@uestc.edu.cn).}}

\maketitle

\begin{abstract}
Microscopic strategy update rules play an important role in the evolutionary dynamics of cooperation among interacting agents on complex networks.
Many previous related works  only consider one \emph{fixed} rule, while in the real world, individuals may switch, sometimes periodically, between rules. It is of particular theoretical interest to investigate under what conditions the periodic switching of strategy update rules facilitates the emergence of cooperation.
To answer this question, we study the evolutionary prisoner's dilemma game on regular networks where agents can periodically switch their strategy update rules. We accordingly develop a theoretical framework of this periodically switched system, where the replicator equation corresponding to each specific microscopic update rule is used for describing the subsystem, and all the subsystems are activated in sequence.
 By utilizing switched system theory, we identify the theoretical condition for the emergence of cooperative behavior.  Under this condition, we  have proved that the periodically switched system with different switching rules can converge to the full cooperation state. Finally, we
 consider an example where two strategy update rules, that is, the  imitation and pairwise-comparison updating, are periodically switched, and find that our numerical calculations validate our theoretical results.
\end{abstract}


\IEEEpeerreviewmaketitle


\section{Introduction}
\IEEEPARstart{C}{ooperative}  behavior  is ubiquitous in nature and human societies~\cite{Hauert2010MIT, Perc2017PR, Vasconcelos2013NCC}. Yet, understanding the evolution and maintenance of cooperation  in a population of self-interested  individuals has been a major scientific puzzle for decades~\cite{ Vincent 2005CUP, Riehl22018ARC, Su2022SA}. Evolutionary game theory provides a powerful mathematical framework to investigate the problem of cooperation~\cite{Hofbauer1998CUP, Vasconcelos2015M3AS, Su2022NHB}, and the prisoner's dilemma game has been widely adopted as a classic paradigm for studying cooperation among selfish individuals~\cite{Nowak1992nature, Perc2008PRE, Perc2006NJP}.


In the context of prisoner's dilemma game, the strategy update rules (i.e., the way in which the strategies of individuals are changed), have a crucial impact on the evolutionary outcomes in a spatially structured population~\cite{ Amaral2018PRE,  Blume1995BEB,  Govaert2021TCNS, Como2020TCNS, Allen2017nature, Banez2021WN, Barreiro16TSMC, Govaert22CSL}. For example, these microscopic rules determine whether cooperation can evolve~\cite{Ohtsuki06Nature, Szab07PR, SzaboPRE98}, and affect the final level of cooperation  in social dilemma of cooperation~\cite{Ramazi2022Auto, Fu2011PRSB, chen2008PRE}. In particular, under a given strategy update rule, Ohtsuki and Nowak  derived a replicator equation on regular networks to study how the fraction of cooperators  changes;   they found that natural selection favors defection under the so-called pairwise-comparison updating in the prisoner's dilemma  game~\cite{Ohtsuki2006JTB}. By contrast, cooperation can emerge if the benefit-to-cost ratio exceeds the degree of the specific interaction network, when an alternative death-birth updating is applied~\cite{Ohtsuki06Nature, Ohtsuki2006JTB}.

Note that many previous works on evolutionary dynamics of cooperation impose the assumption that the strategy update rule is fixed and time-invariant\cite{Shi2020TNSE, Zhang2019IEEETCS, Hu2021TNSE}. However, this assumption is not always realistic and merely represents an oversimplification of reality. Indeed, during the evolutionary process individuals can use different strategy update rules or switch their update rule to another. Due to periodic changes in the environment~\cite{Fleming1986JM, Hutto1981Auk, Conner11981Auk} or driven by their periodic behavior modes~\cite{Clark1980book, Szolnoki2013SR, Karlen1994AA},  they may periodically switch their strategy update rules. In addition, the basic idea of switching has been introduced into evolutionary games~\cite{Wang2020TNSE,  Liberzon2005BB, Cunha2021ECC}. For example, Hilbe  \emph{et  al.}  studied evolutionary dynamics of cooperation in stochastic games where the game structure can be periodically switched~\cite{Hilbe2018Nature}. Subsequently, Su  \emph{et  al.}  studied evolutionary dynamics on complex networks where game transitions happen~\cite{Su19PNAS}. Very recently, Shu and Fu studied replicator dynamics with feedback-evolving games, where periodic switching of two different game matrices is considered~\cite{ShuPRSA2022}. Also noticeable, Li  \emph{et  al.}  studied evolutionary dynamics of cooperation on temporal networks, where periodic switching of structured populations can happen~\cite{LiNC2020}. However, far fewer works have studied the evolutionary dynamics of cooperation in the scenario of periodic switching of update rules. For example, it is still unclear whether cooperative behavior can emerge in the networked prisoner's dilemma game when periodic switching of update rules is considered. Furthermore, given a tentative positive answer to the possibility question, one still wants to explore further what the mathematical conditions are  for the emergence of cooperative behavior among interacting individuals.

In order to answer the questions mentioned above, in this paper, we study the evolutionary dynamics of cooperation
in the networked prisoner's dilemma game with periodic switching strategy update rules. By means of the standard form of replicator equation on graphs for the prisoner's dilemma game, we formulate our mathematical model of switched system for the periodic switching strategy update rules. By using switched system theory,  we derive a theoretical condition for the emergence of cooperative behavior  on graphs.
Under this condition, we show that  the periodically switched system with different switching update rules  can evolve into a full cooperation state. Finally,  as a supporting example, we consider that individuals can periodically switch the update rules between the classic  imitation (IM)  and pairwise-comparison (PC) updating in the networked prisoner's dilemma game, and find that our numerical calculations verify our theoretical results. The main contributions of our work are listed as follows:
\begin{itemize}
    \item  We formulate a mathematical framework to study the evolutionary dynamics of cooperation with periodic switching strategy update rules.
       \item We theoretically prove that the periodically switched system can reach a full cooperation state. In addition, we identify the mathematical condition for the emergence of cooperative behavior.
    \item We consider an example where two strategy update rules, that is, the IM and PC updating, are periodically switched. In this example, we find that the final evolutionary outcome can support the theoretical condition for the emergence of cooperation we obtained.
    \end{itemize}
The rest of this paper is organized as follows. In Section II, we formalize the problem. Here, we first describe the networked prisoner's dilemma game. Then, we depict the replicator equation with a strategy update rule  on regular networks.  Afterwards we characterize the periodically switched system with different strategy update rules. In Section III, we present the theoretical results, including the existence and (asymptotic) stability analysis of equilibria.
 In Section IV, we present numerical results to verify our obtained theoretical results.
 Finally, conclusions are summarized  in Section V.

\textbf{Notation:}
Throughout the paper, we denote the set of nonnegative  and strictly positive integer numbers by $\mathbb{N}$ and $\mathbb{N_{+}}$, respectively. Let $\mathbb{R}=(-\infty, +\infty)$ and
$\mathcal{M}=\{1,  2,  \ldots,  m\}$,  where $m\geq 2$ is a fixed positive integer.   For a number $x\in \mathbb{R}$, $\left| x\right|$ denotes its absolute value. $\mathop{E(X)}$ represents the expected value of a random variable $\mathop{X}$.

\section{PROBLEM FORMULATION}
\subsection{Networked Prisoner's Dilemma Game}
We consider a structured population of $n\in\mathbb{N_{+}}$ players, whose interaction structure is characterized by a regular network with a general degree $k>2$ ($k\in\mathbb{N_{+}}$).
In this graph, vertices represent the  agents and the edges show who interacts with whom. In each round, each player $i$ plays the evolutionary prisoner's dilemma game with its neighbors, and can choose to cooperate ($C$) or defect ($D$).  We consider the payoff matrix for the game as
\begin{equation}
A=\bordermatrix{
      &C  &D  \cr
C &b-c & -c\cr
D &b & 0 },
\label{matrix}
\end{equation}
where $b$ represents the benefit of cooperation and $c$ ($0<c<b$) represents the cost of cooperation. After engaging in the pairwise interactions with all the adjacent neighbors, each player reaps its accumulated payoff $\pi_i$ determined by the payoff matrix. We denote the fitness (i.e., the reproductive rate) of this player by $g_i$, and it is defined to be  $g_i = 1-\omega+\omega \pi_i$, where $\omega\in[0, 1]$ measures the strength of selection~\cite{Ohtsuki06Nature}.

\begin{assumption}\label{assumption1}
The network considered here is \emph{connected} and its topology is a regular network with the degree $k>2$.
\end{assumption}Assumption \ref{assumption1} implies that the $k$-regular network cannot be disconnected, and this specific network topology are not valid for all $k$-regular networks, such as  ring and tree~\cite{Barab12NS}, since $k>2$.

\subsection{Replicator Equation with a Strategy Update Rule}
According to the evolutionary selection principle, players in the population adjust their strategy from time to time, and the way how to do it may influence the final evolutionary outcome significantly. Here, we consider $m\geq 2$ $(m\in\mathbb{N_{+}})$ different strategy update rules from the set $\mathcal{M}$.
From the obtained form of the replicator equation for prisoner's dilemma game on graphs in Ref.  \cite{Ohtsuki2006JTB}, the dynamical equation for the strategy update rule $i$ ($i\in \mathcal{M}$) can be given as
\begin{equation}\label{eq2}
\begin{aligned}
\frac{\textrm{d}x(t)}{\textrm{d}t}=f_{i}(x(t))=\alpha_{i}(\omega, k, b, c)x(t)(1-x(t)),
\end{aligned}
\end{equation}
where $x(t)\in [0, 1]$ is the fraction of $C$-players (i.e., cooperators) in the population,  and $\alpha_{i}(\omega, k, b, c)$ is a function of selection strength $\omega$, the network degree $k$, the cooperation benefit $b$ and the cooperation cost $c$. Here, we denote by $x(0)=x_0$ the initial fraction of cooperators in the population. We note that using the pair-approximation approach, the form of Eq.~(\ref{eq2}) can be obtained for a series of microscopic update rules under weak selection (i.e., $\omega \rightarrow 0$) when the payoff matrix of the game is described by Eq.~(\ref{matrix}). Accordingly, Eq.~(\ref{eq2}) can be used to describe the evolutionary dynamics of cooperation on a regular network when the microscopic update rule is fixed during the evolutionary process.
\begin{assumption}\label{assumption2}
The population structure, the payoff matrix $A$, and the strength of selection $\omega$ are fixed.
\end{assumption}
Assumption \ref{assumption2}  implies  that the considered parameters (i.e.,  $n$, $\omega$, $k$, $b$ and $c$) do not change over time, that is, they are fixed.  Therefore,   the function $\alpha_i(\omega, k, b, c)$ can be simplified to be $\alpha_i$ since it is now just a constant.

In this case, Eq. \eqref{eq2} has two equilibria, which are $x^*=0$ and $x^*=1$. If $\alpha_{i}>0$ and  the initial state $x_0\in (0, 1)$, then $\frac{\textrm{d}x(t)}{\textrm{d}t}>0$ for all $t\geq0$. Note that this implies  that the equilibrium point $x^*=0$ is unstable and $x^*=1$ is (asymptotically) stable. If $\alpha_{i}<0$  and $x_0\in (0, 1)$,  then $\frac{\textrm{d}x(t)}{\textrm{d}t}<0$ for all $t\geq0$, and the equilibrium point $x^*=0$ is (asymptotically) stable and $x^*=1$ is unstable.

\begin{figure}[!t]
\centering
\includegraphics[width=3in]{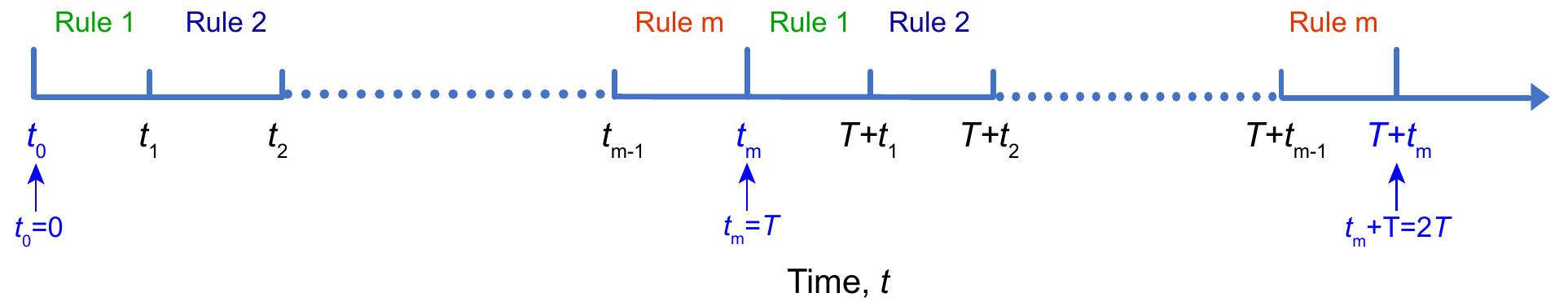}
\caption{Schematic diagram of periodic switching of multiple strategy update rules over time. Rule 1 which is active in time interval $[t_0+\theta T, t_1+\theta T)$ changes to rule 2 at time $t_1+\theta T$, where $\theta \in \mathbb{N}$.  Similarly, rule $i$ activated in time interval $[t_{i-1}+\theta T, t_i+\theta T)$ changes to rule $i+1$ at time $t_i+\theta T$ for any $i\in \mathcal{W}$. Rule $m$ activated in time interval $[t_{m-1}+\theta T, (\theta+1) T)$ changes to rule $1$ at time $(\theta+1) T$.}
\label{fig1}
\end{figure}


\subsection{Periodically Switched System with  Different Strategy Update Rules}
In fact, individuals may adopt more than one strategy update rule, and  may periodically use different strategy update rules. Accordingly, the evolutionary dynamics of cooperation in this scenario where multiple update rules are periodically switched can be described by
\begin{equation}\label{eq3}
\begin{aligned}
\frac{\textrm{d}x(t)}{\textrm{d}t}=f_{\sigma(t)}(x(t)),
\end{aligned}
\end{equation}
where $\sigma(t)$ is a periodic switching signal, given by
\begin{equation}\label{eq4}
\begin{aligned}
\sigma (t)=i,   \,\,\,\,t\in\Omega_{i\theta}.
\end{aligned}
\end{equation}
Here $\Omega_{i \theta}= [\theta T + t_{i-1}, \theta T + t_i)$, $\theta \in \mathbb{N}$, $i \in \mathcal{M}$.
  $t_{0}=0$ and $t_{m}=T$ is the period of the periodic switching rule sequence, where $m$ is the number of all subsystems with distinct strategy update rules, and each subsystem is described by Eq.~\eqref{eq2} under Assumption \ref{assumption2}. The switching signal
 $\sigma(t)$:$[0,\infty)$$\mapsto$$ \mathcal{M}$ is assumed to be a piecewise constant function continuous from the right. This implies that subsystem $i$ under  the $i$-th switching rule is activated over the time interval $\Omega_{i\theta}$, and the subsystem $i+1$ under the $(i+1)$-th switching rule will be activated at $ \theta T+t_{i}$  for any $i\in \mathcal{M}$ (see Fig. \ref{fig1}). Denote the activating period of subsystem $i$  by $\Delta t_{i}=(\theta T+t_{i})-(\theta T+t_{i-1})=t_i - t_{i-1}$. Accordingly, $t_{\rho}=\sum\limits_{i=1}^{\rho}\Delta t_{i}$ is the $\rho$-th switching instant in the first period (i.e., $\theta=0$), and the period $T$ is $T=\sum\limits_{i=1}^{m}\Delta t_{i}$. For convenience, we use $\bar{\Omega}_\theta$ to represent the  period $[ \theta T, (\theta+1) T)$, and thus $\bar{\Omega}_\theta=[ \theta T, (\theta+1) T)=\bigcup\limits_{i=1}^{m}\Omega_{i \theta}$, for any $\theta \in \mathbb{N}.$



From the above description, one can observe that  the periodically switched system \eqref{eq3} is not only affected by the activation sequential order of $m$ subsystems with distinct switching rules, but also by the switching instant (i.e., $\theta T+t_{i},  \theta \in\mathbb{N}$) between two subsystems. Subsequently,  we give the definition of $m$-ary sequence set $\mathcal{A}$ for arranging all possible sequential order of all subsystems considered over the period $\bar{\Omega}_\theta$, and determine the specific expression of \eqref{eq3} for any order of the activation sequence.



\begin{definition}
\label{definition1}
\begin{equation}\label{eq5}
\begin{aligned}\mathcal{A}=\big\{(\alpha_{\lambda_1}, \dots, \alpha_{\lambda_j}, \dots, \alpha_{\lambda_m})\big| (\lambda_1, \dots, \lambda_j, \dots, \lambda_m)\in \mathcal{B}\big\}.
\end{aligned}
\end{equation}
Here, $\mathcal{B}$ represents all the possible permutations of the set $\mathcal{M}$,  and the sequence $(\alpha_{\lambda_1}, \dots, \alpha_{\lambda_j}, \dots, \alpha_{\lambda_m})$ denotes the activation sequence, where $j\in\mathcal{W}$ and $\mathcal{W}=\{2, \dots, m-1\}$.
\end{definition}

For any activation sequence $(\alpha_{\lambda_1}, \dots, \alpha_{\lambda_j}, \dots, \alpha_{\lambda_m})\in \mathcal{A}$ of subsystems considered at each period $\bar{\Omega}_\theta$,  we know that  subsystem $\frac{\textrm{d}x(t)}{\textrm{d}t}=\alpha_{\lambda_1}x(t)(1-x(t))$ is activated in $\Omega_{1\theta}$, subsystem $\frac{\textrm{d}x(t)}{\textrm{d}t}=\alpha_{\lambda_j}x(t)(1-x(t))$ is activated in $\Omega_{j\theta}$, until subsystem $\frac{\textrm{d}x(t)}{\textrm{d}t}=\alpha_{\lambda_m}x(t)(1-x(t))$ is activated in $\Omega_{m\theta}$.  In this case, the periodically switched system \eqref{eq3} is rewritten as
\begin{equation}\label{eq6}
\frac{\textrm{d}x(t)}{\textrm{d}t}\bigg|_{(\alpha_{\lambda1}, \dots, \alpha_{\lambda j}, \dots, \alpha_{\lambda_m})}=\left\{
\begin{aligned}
\alpha_{\lambda_1}&x(t)(1-x(t)),& &t\in\Omega_{1\theta},\\
&\vdots&\\
\alpha_{\lambda_j}&x(t)(1-x(t)),& &t\in\Omega_{j\theta},\\
&\vdots&\\
\alpha_{\lambda_m}&x(t)(1-x(t)),& &t\in\Omega_{m\theta}.
\end{aligned}
\right.
\end{equation}
For convenience, we denote the system above by $\Sigma|_{(\alpha_{\lambda_1}, \dots, \alpha_{\lambda_j}, \dots, \alpha_{\lambda_m})}$, and  we have
\begin{equation}\label{eq7}
\begin{aligned}
x(t)=\frac{1}{1+\frac{1-x_0}{x_0}e^{-\Lambda_{i\theta}(t)}}
\end{aligned}
\end{equation}
for all $t\in\Omega_{i\theta}$, where $\Lambda_{1\theta}(t)=\theta \Big[\sum\limits_{i=1}^{m-1}(\alpha_{\lambda_i}-\alpha_{\lambda_{i+1}})t_{i}+(\alpha_{\lambda_m}-\alpha_{\lambda_1})T\Big]+\alpha_{\lambda_1}t$,
$\Lambda_{j\theta}(t)= \theta \Big[\sum\limits_{i=1}^{m-1}(\alpha_{\lambda_i}-\alpha_{\lambda_{i+1}})t_{i}+(\alpha_{\lambda_m}-\alpha_{\lambda_1})T\Big]+\sum\limits_{l=1}^{j-1}(\alpha_{\lambda_l}-\alpha_{\lambda_{l+1}})( \theta T+t_{l})+\alpha_{\lambda_j}t$,  $j\in\mathcal{W}$, and $\Lambda_{m\theta}(t)=( \theta +1)\sum\limits_{i=1}^{m-1}(\alpha_{\lambda_i}-\alpha_{\lambda_{i+1}})t_{i}+\beta_{m}t$.

Accordingly,
 the periodically switched system set under the activation sequence $\mathcal{A}$ over the period  $\bar{\Omega}_\theta$ is
 \begin{equation}\label{eq8}
 \begin{aligned}
\Delta=\big\{\Sigma|_{(\alpha_{\lambda_1}, \dots, \alpha_{\lambda_j}, \dots, \alpha_{\lambda_m})}\big|\,\,(\lambda_1, \dots, \lambda_j, \dots, \lambda_m)\in \mathcal{B}\big\}.
\end{aligned}
\end{equation}

\begin{figure}[!t]
\centering
\includegraphics[width=3.5in]{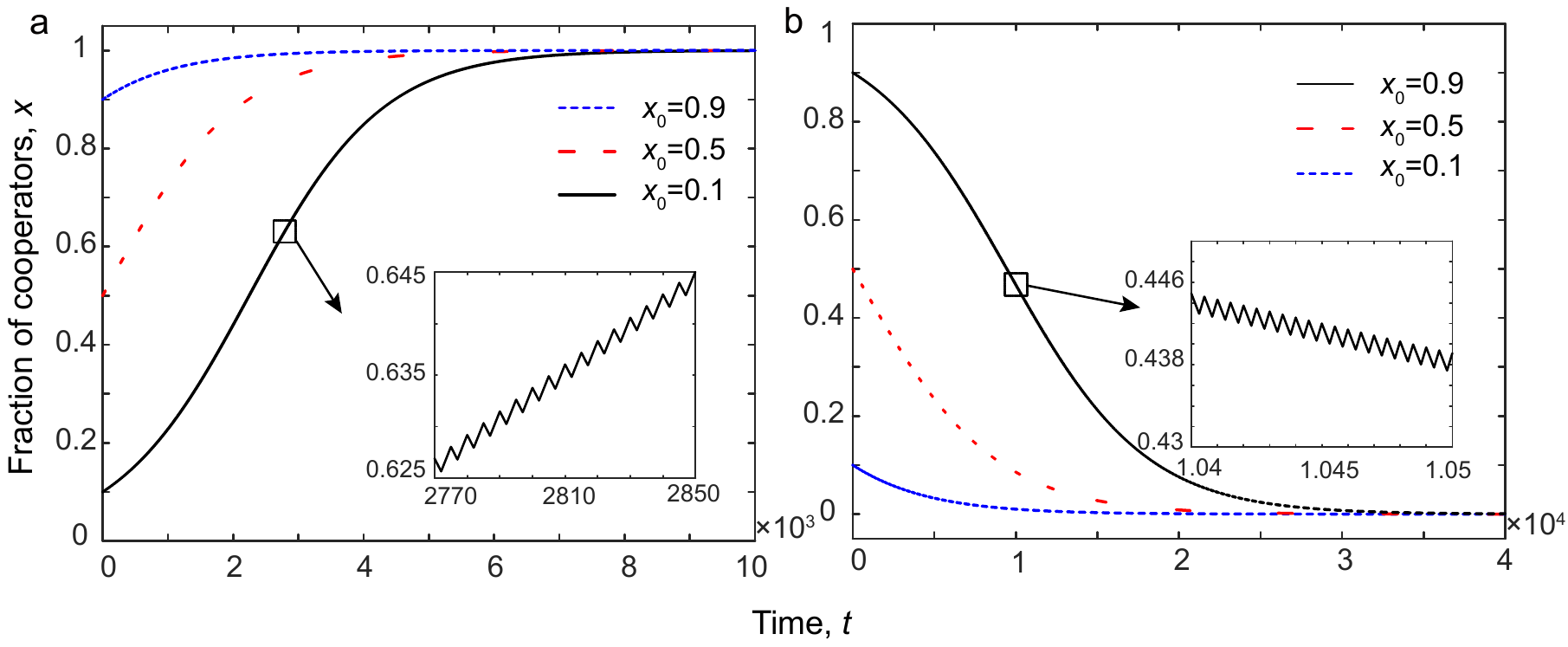}
\caption{Time evolution of the fraction of cooperators for  ``first PC updating, then IM updating''  at each period $\bar{\Omega}_\theta$  under different initial states. Three different initial states are $x_0=0.1$,  $x_0=0.5$, and $x_0=0.9$, respectively.
The switching point is $t_{1}=2$ for panel $a$, and $t_{1}=3$ for panel $b$.
Other parameters: $\omega=0.01$, $b=2$, $c=0.2$, $k=4$, and $T=5$.}
\label{fig2}
\end{figure}

For each  subsystem $i$ (i.e., $\frac{\textrm{d}x(t)}{\textrm{d}t}=\alpha_{\lambda_i}x(t)(1-x(t)), t\in\Omega_{i\theta}$), we have the following two remarks.
\begin{Remark}
\label{Remark1}
Under Assumption~\ref{assumption2}, the subsystem $i$ has two equilibria, which are $x^*=0$  and $x^*=1$. If  $\alpha_{\lambda_i}>0$ and $x_0\in (0, 1)$, then $x^*=0$ is unstable and $x^*=1$ is (asymptotically) stable.
If $\alpha_{\lambda_i}<0$  and $x_0\in (0, 1)$, then $x^*=0$ is (asymptotically) stable and $x^*=1$ is unstable.
\end{Remark}

\begin{Remark}
\label{Remark2}
The trajectory of each subsystem $i$ in each time interval $\Omega_{i\theta}$,  $\theta \in \mathbb{N}$ is either monotonically increasing or decreasing. If $\alpha_{\lambda_i}>0$ and $x_0\in (0, 1)$, $\frac{\textrm{d}x(t)}{\textrm{d}t}>0$ for all $t\in\Omega_{i\theta}$ and thus the subsystem $i$ is  monotonically increasing for all $t\in\Omega_{i\theta}$. While if $\alpha_{\lambda_i}<0$, then $\frac{\textrm{d}x(t)}{\textrm{d}t}<0$ for all $t\in\Omega_{i\theta}$ and thus the subsystem $i$  is  monotonically decreasing for all $t\in\Omega_{i\theta}$.
\end{Remark}

In this work, we aim to explore the mathematical condition of all periodic  switching instants for any  activation sequence, in which the full cooperation state can be reached. To  this end,  we impose the following standing assumption.
\begin{assumption}\label{assumption3}
There exists at least one  subsystem  with a switching rule which renders the equilibrium point $x^*=1$ stable, that is, ${\exists} \,\, j\in \mathcal{M}$, s.t. $\alpha_{\lambda_j}>0$.
\end{assumption}
One can observe that this assumption ensures that it is possible to render the cooperative equilibrium state $x^*=1$ asymptotically stable by  reasonably adjusting the activation time (e.g., by always only activating the subsystem where $x^*=1$ is the asymptotically stable equilibrium point).


\section{Theoretical Results}
Without loss of generality, we carry out the theoretical analysis on any  periodically switched system $\Sigma|_{(\alpha_{\lambda_1}, \dots, \alpha_{\lambda_j}, \dots, \alpha_{\lambda_m})}\in \Delta$ with the activation sequence $(\alpha_{\lambda_1}, \dots, \alpha_{\lambda_j}, \dots, \alpha_{\lambda_m})$, which is described by \eqref{eq6}.
Moreover,  it is clear from  Remark \ref{Remark2} that  the extremum (maximum or minimum) of a trajectory of \eqref{eq6} over the period $\bar{\Omega}_\theta$ is obtained either  at the boundary or at the switching time, i.e., $\theta T+t_{i}, \theta \in\mathbb{N}$.  Then,  we give a lemma to illustrate the relationship between the activation sequence $(\alpha_{\lambda_1}, \dots, \alpha_{\lambda_j}, \dots, \alpha_{\lambda_m})$ and the sequence $\{x(\theta T+t_{i})\}_{\theta\in \mathbb{N}}$.



\begin{lemma}\label{lemma1}
For all $v\in\mathcal{V}$ (i.e., $\mathcal{V}=\{0, \dots, m-1\}$) and $\theta \in \mathbb{N}$, the following statements hold:
\begin{enumerate}
 \item \begin{equation}\label{eq9}
\begin{aligned}
\sum\limits_{i=1}^{m} \alpha_{\lambda_i}(t_{i}-t_{i-1})>0
\end{aligned}
\end{equation} if and only if
\begin{equation}\label{eq10}
\begin{aligned}
x(\theta T+t_{v})<x((\theta +1)T+t_{v});
\end{aligned}
\end{equation}
 \item \begin{equation}\label{eq11}
\begin{aligned}
\sum\limits_{i=1}^{m} \alpha_{\lambda_i}(t_{i}-t_{i-1})<0
\end{aligned}
\end{equation} if and only if
\begin{equation}\label{eq12}
\begin{aligned}
x(\theta T+t_{v})>x((\theta +1)T+t_{v}).
\end{aligned}
\end{equation}
 \end{enumerate}
\end{lemma}
\begin{proof}
For any $v\in\mathcal{V}$,  define the difference between $x(\theta T+t_{v})$ and $x((\theta +1)T+t_{v})$ to be
\begin{equation}\label{eq13}
\begin{aligned}
\mathcal{G}(v)&=x(\theta T+t_{v})-x((\theta +1)T+t_{v})\\
&=\frac{1}{1+\frac{1-x_0}{x_0}e^{-\Lambda_{i\theta}(\theta T+t_{v})}}-\frac{1}{1+\frac{1-x_0}{x_0}e^{-\Lambda_{i(\theta+1)}[(\theta +1)T+t_{v}]}}\\
&=\frac{\frac{1-x_0}{x_0}\{e^{-\Lambda_{i(\theta+1)}[(\theta +1)T+t_{v}]}-e^{-\Lambda_{i\theta}(\theta T+t_{v})}\}}{\{1+\frac{1-x_0}{x_0}e^{-\Lambda_{i\theta}(\theta T+t_{v})}\}\{1+\frac{1-x_0}{x_0}e^{-\Lambda_{i(\theta+1)}[(\theta +1)T+t_{v}]}\}}\\
&=\mathcal{L} \mathcal{F}(e^{\mathcal{T}_{i\theta}(v)}-1),
\end{aligned}
\end{equation}
where $\mathcal{L}=\frac{e^{\Lambda_{i(\theta+1)}[(\theta +1)T+t_{v}]}}{\{1+\frac{1-x_0}{x_0}e^{-\Lambda_{i\theta}(\theta T+t_{v})}\}\{1+\frac{1-x_0}{x_0}e^{-\Lambda_{i(\theta+1)}[(\theta +1)T+t_{v}]}\}} $,  $\mathcal{F}=\frac{1-x_0}{x_0}\frac{1}{e^{\Lambda_{i(\theta+1)}[(\theta +1)T+t_{v}]+\Lambda_{i\theta}(\theta T+t_{v})}}$, and $\mathcal{T}_{i\theta}(v)=\Lambda_{i\theta}(\theta T+t_{v})-\Lambda_{i(\theta+1)}[(\theta +1)T+t_{v}]$.
Since the exponential function $h(u)=e^{u}$ is monotonically increasing  for $u\in\mathbb{R}$, and is positive (i.e.,  $h(u)>0$)  for all $u\in\mathbb{R}$,  one can check that  $\mathcal{L}>0$ and $\mathcal{F}>0$ when the initial state  $x_0$ is not any of the equilibria. Therefore, the sign of $\mathcal{G}(v)$ is  equal to that of $\mathcal{T}_{i\theta}(v)$.

Next, we present the specific expression of the function $\mathcal{T}_{i\theta}(v)$ for any $v\in\mathcal{V}$. For $v=0$, $\mathcal{T}_{i\theta}(v)$ is given by
\begin{equation}\label{eq14}
\begin{aligned}
\mathcal{T}_{i\theta}(v)&=\mathcal{T}_{1\theta}(0)\\
&=\Lambda_{1\theta}(\theta T)-\Lambda_{1(\theta+1)}((\theta +1)T)\\
&=\bigg\{\theta \Big[\sum_{i=1}^{m-1}(\alpha_{\lambda_i}-\alpha_{\lambda_{i+1}})t_{i}+(\alpha_{\lambda_m}-\alpha_{\lambda_1})T\Big]+\alpha_{\lambda_1}(\theta T)\bigg\}\\
&\quad-\bigg\{(\theta+1) \Big[\sum_{i=1}^{m-1}(\alpha_{\lambda_i}-\alpha_{\lambda_{i+1}})t_{i}+(\alpha_{\lambda_m}-\alpha_{\lambda_1})T\Big]\\
&\quad+\alpha_{\lambda_1}[(\theta+1) T]\bigg\}\\
&=-\sum_{i=1}^{m-1}(\alpha_{\lambda_i}-\alpha_{\lambda_{i+1}})t_{i}-\alpha_{\lambda_m}T\\
&=-\sum_{i=1}^{m-1}(\alpha_{\lambda_i}-\alpha_{\lambda_{i+1}})t_{i}-\alpha_{\lambda_m}T+\alpha_{\lambda_1}t_{0}\\
&=-\sum\limits_{i=1}^{m} \alpha_{\lambda_i}(t_{i}-t_{i-1}).
\end{aligned}
\end{equation}

For $v\in\{1,2,\dots, m-2\}$, $\mathcal{T}_{i\theta}(v)$  is given by
\begin{equation}\label{eq15}
\begin{aligned}
\mathcal{T}_{i\theta}(v)&=\mathcal{T}_{(v+1)\theta}(v)\\
&=\Lambda_{(v+1)\theta}(\theta T+t_{v})-\Lambda_{(v+1)(\theta+1)}((\theta +1)T+t_{v})\\
&=\bigg\{ \theta \Big[\sum_{i=1}^{m-1}(\alpha_{\lambda_i}-\alpha_{\lambda_{i+1}})t_{i}+(\alpha_{\lambda_m}-\alpha_{\lambda_1})T\Big]\\
&\quad+\sum_{l=1}^{v}(\alpha_{\lambda_l}-\alpha_{\lambda_{l+1}})( \theta T+t_{l})+\alpha_{\lambda_{v+1}}(\theta T+t_{v})\bigg\}\\
&\quad-\bigg\{(\theta+1) \Big[\sum_{i=1}^{m-1}(\alpha_{\lambda_i}-\alpha_{\lambda_{i+1}})t_{i}+(\alpha_{\lambda_m}-\alpha_{\lambda_1})T\Big]\\
&\quad+\sum_{l=1}^{v}(\alpha_{\lambda_l}-\alpha_{\lambda_{l+1}})\Big[(\theta+1) T+t_{l}\Big]\\
&\quad+\alpha_{\lambda_{v+1}}\Big[(\theta +1)T+t_{v}\Big]\bigg\}\\
&=-\sum_{i=1}^{m-1}(\alpha_{\lambda_i}-\alpha_{\lambda_{i+1}})t_{i}-\alpha_{\lambda_m}T\\
&=-\sum_{i=1}^{m-1}(\alpha_{\lambda_i}-\alpha_{\lambda_{i+1}})t_{i}-\alpha_{\lambda_m}T+\alpha_{\lambda_1}t_{0}\\
&=-\sum\limits_{i=1}^{m} \alpha_{\lambda_i}(t_{i}-t_{i-1}).
\end{aligned}
\end{equation}

For $v=m-1$, $\mathcal{T}_{i\theta}(v)$  is given by
\begin{equation}\label{eq16}
\begin{aligned}
\mathcal{T}_{i\theta}(v)&=\mathcal{T}_{m\theta}(m-1)\\
&=\Lambda_{m\theta}(\theta T+t_{m-1})-\Lambda_{m(\theta+1)}((\theta +1)T+t_{m-1})\\
&=\Big[(\theta +1)\sum_{i=1}^{m-1}(\alpha_{\lambda_i}-\alpha_{\lambda_{i+1}})t_{i}+\alpha_{\lambda_m}(\theta T+t_{m-1})\Big]\\
&\quad-\bigg\{(\theta +2)\sum_{i=1}^{m-1}(\alpha_{\lambda_i}-\alpha_{\lambda_{i+1}})t_{i}\\
&\quad+\alpha_{\lambda_m}\Big[(\theta +1)T+t_{m-1}\Big]\bigg\}\\
&=-\sum_{i=1}^{m-1}(\alpha_{\lambda_i}-\alpha_{\lambda_{i+1}})t_{i}-\alpha_{\lambda_m}T\\
&=-\sum_{i=1}^{m-1}(\alpha_{\lambda_i}-\alpha_{\lambda_{i+1}})t_{i}-\alpha_{\lambda_m}T+\alpha_{\lambda_1}t_{0}\\
&=-\sum\limits_{i=1}^{m} \alpha_{\lambda_i}(t_{i}-t_{i-1}).
\end{aligned}
\end{equation}

Therefore,  for all  $v\in\mathcal{V}$ and  $\theta \in \mathbb{N}$, it holds that
\begin{equation}\label{eq17}
\begin{aligned}
\mathcal{T}_{i\theta}(v)=-\sum\limits_{i=1}^{m} \alpha_{\lambda_i}(t_{i}-t_{i-1}).
\end{aligned}
\end{equation}
 From \eqref{eq17}, one knows that the sign of $\sum\limits_{i=1}^{m} \alpha_{\lambda_i}(t_{i}-t_{i-1})$  is the opposite of
  $\mathcal{G}(v)$ since the sign of $\mathcal{G}(v)$ is equal to that of $\mathcal{T}_{i\theta}(v)$.

Moreover, we discuss the cases of $\eqref{eq9}\Leftrightarrow \eqref{eq10}$ and $\eqref{eq11}\Leftrightarrow \eqref{eq12}$.
We first prove $\eqref{eq9}\Rightarrow \eqref{eq10}$.
If  $\eqref{eq9}$ is satisfied, one can get $\mathcal{G}(v)<0$ and thus $\eqref{eq10}$ holds. We further prove $\eqref{eq9}\Leftarrow \eqref{eq10}$. If $\eqref{eq10}$ holds, then $\mathcal{G}(v)<0$. This leads to $\sum\limits_{i=1}^{m}\alpha_{\lambda_i}(t_{i}-t_{i-1})>0$ and $\eqref{eq9}$ holds. In a word, $\eqref{eq9}\Leftrightarrow \eqref{eq10}$.  Similarly, the case $\eqref{eq11}\Leftrightarrow \eqref{eq12}$ can be proved.
\end{proof}




Lemma \ref{lemma1} illustrates the relationship between the activation ordering of $m$ subsystems with distinct switching rules $(\alpha_{\lambda_1}, \dots, \alpha_{\lambda_j}, \dots, \alpha_{\lambda_m})$ and the switching instants between two subsystems ${x(\theta T+t_i)}_{\theta\in \mathbb{N}}$. To be specific, under $\eqref{eq9}$, the sequence $\{x(\theta T+t_{i})\}_{\theta\in \mathbb{N}}$ is  monotonically increasing, while under $\eqref{eq11}$,  the sequence $\{x(\theta T+t_{i})\}_{\theta\in \mathbb{N}}$ is monotonically decreasing. Accordingly, we can then respectively discuss the existence and (asymptotic) stability of equilibria corresponding to the cases of $\eqref{eq9}$ and $\eqref{eq11}$, shown in the next subsections.

For simplicity, let $d_{\theta}=\mathop{\inf} \limits_{t\in\bar{\Omega}_\theta}x(t)$ denote the infimum value over the interval $\bar{\Omega}_\theta$ and let $\{d_{\theta}\}_{\theta \in \mathbb{N}}$ denote a sequence of the infimum values. Furthermore, let $y_{\theta}=\mathop{\sup} \limits_{t\in\bar{\Omega}_\theta}x(t)$ denote the supremum value over the interval $\bar{\Omega}_\theta$ and let $\{y_{\theta}\}_{\theta \in \mathbb{N}}$ denote a sequence of the supremum values.
\begin{Remark}
\label{Remark3}
 Under $\eqref{eq9}$, the sequence $\{d_{\theta}\}_{\theta \in \mathbb{N}}$ is strictly monotonically increasing for all  $\theta \in \mathbb{N}$, that is, $d_{\theta}>d_{\theta+1}$, $\forall \,\, \theta \in \mathbb{N}$.
 Under $\eqref{eq11}$, the sequence $\{y_{\theta}\}_{\theta \in \mathbb{N}}$ is strictly monotonically decreasing for all  $\theta \in \mathbb{N}$, that is, $y_{\theta}<y_{\theta+1}$, $\forall \,\, \theta \in \mathbb{N}$.
\end{Remark}


\subsection{Existence of the Equilibrium Point}

To start, we give the definition on the existence of the equilibrium point.

\begin{theorem}
\label{theorem1}
If   \eqref{eq9} or \eqref{eq11} holds (i.e., $\sum\limits_{i=1}^{m} \alpha_{\lambda_i}(t_{i}-t_{i-1})\neq0$),  there always exist two equilibria of the periodically switched system \eqref{eq6}, namely $x^\ast=1$ and $x^\ast=0$.
\end{theorem}
\begin{proof}
 Solving  $\dot{x}(t)|_{(\alpha_{\lambda_1}, \alpha_{\lambda_2}, \dots, \alpha_{\lambda_m})}=0$, it is easy to check that  the system has two equilibria on the boundary of the parameter space, namely, $x^\ast=1$ and  $x^\ast=0$. We now proceed to explore whether there is an interior equilibrium point. It is equivalent to considering  that  if $x(0)=x^*\in(0, 1)$, then $x(t)=x^*$ for all $t>0$.  However, when the initial state $x(0)$ is in the subsystem $i\in \mathcal{M}$,  one knows that the trajectory  of $x(t)$ is either monotonically increasing or decreasing  in $\Omega_{i\theta}$. This implies that  $x(t)$ is  not  always constant and therefore not identically  $x^*$. Thus,  the interior equilibrium point does not exist.
\end{proof}

\subsection{Stability of the Equilibrium Point}
Under the condition of \eqref{eq9},  we separately present the stability analysis of the obtained equilibria (i.e., $x^\ast=0$ and $x^\ast=1$) in Theorem \ref{theorem2}, and the asymptotic stability analysis of the equilibrium point  $x^\ast=1$ in Theorem \ref{theorem3}.

\begin{theorem}\label{theorem2}
If \eqref{eq9} holds,   the equilibrium point $x^\ast=0$ of the periodically switched system \eqref{eq6}   is unstable and $x^\ast=1$ is stable.
\end{theorem}
\begin{proof}
Firstly, we prove that  the equilibrium point $x^\ast=0$ of \eqref{eq6}  is unstable. From Remark  \ref{Remark3},   the fact that under  \eqref{eq9} for all $\theta\in \mathbb{N}$
the sequence $\{d_{\theta}\}_{\theta \in \mathbb{N}}=\{\mathop{\min} \limits_{t\in\bar{\Omega}_\theta}x(t)\}_{\theta \in \mathbb{N}}$ is strictly monotonically increasing  implies $d_{0}=\mathop{\min} \limits_{t \ge 0}x(t)$.
Let $\epsilon_0=d_0$, then for all $\delta>0$,  if $\left|x_0\right|<\delta$, we have $\left|x(t)\right|>\epsilon_0$ for all $t \ge 0$. Therefore,   the equilibrium point $x^\ast=0$ is unstable.

Next, we prove that the equilibrium point $x^\ast=1$ of \eqref{eq6}  is stable from three aspects, where the infimum value of \eqref{eq6}  over the period  $\bar{\Omega}_\theta$ is taken at $\theta T$, $\theta T+t_{1}$, and $\theta T+t_{j}$, that is, $d_{\theta}=x(\theta T)$, $d_{\theta}=x(\theta T+t_{1})$, and $d_{\theta}=x(\theta T+t_{j})$ for $j\in\mathcal{W}$ and $\theta\in \mathbb{N}$. Based on the definition of Lyapunov stability~\cite{Khalil2002}, for any $\epsilon>0$, we shall find $\delta>0$, such that $\left|x(t)-1\right|<\epsilon$ holds for $t\ge 0$.

For $d_{\theta}=x(\theta T)$,
from the definition of $d_{\theta}$ it means that the infimum value of \eqref{eq6}  over the period  $\bar{\Omega}_\theta$ is taken  at $\theta T$.
Then according to Remark \ref{Remark3},  one yields $d_0=\mathop{\inf}\limits_{t\in\bar{\Omega}_\theta}x(t)=\mathop{\inf}\limits_{t\in\bar{\Omega}_0}x(t)=x(0)=x_0=\mathop{\min} \limits_{t \ge 0} x(t)$.  From $x(t)\in[0, 1]$, it follows  that
 \begin{equation}\label{eq18}
\begin{aligned}
\mathop{\max} \limits_{t \ge 0} \left|x(t)-1\right|&=\mathop{\max} \limits_{t \ge 0} \{1-x(t)\}=1-\mathop{\min} \limits_{t \ge 0} x(t)=1-x_0.
\end{aligned}
 \end{equation}
 Thus, if $\mathop{\max} \limits_{t \ge 0} \left|x(t)-1\right|=1-x_0=\left|x_0-1\right|<\epsilon$, then  $\left|x(t)-1\right|<\epsilon$.
This  indicates that for any $\epsilon>0$, there exists a $\delta=\epsilon>0$ such that $\left|x(t)-1\right|< \epsilon$ for all  $t\geq 0$, whenever $\left|x_0-1\right|<\delta$.

For $d_{\theta}=x(\theta T+t_{1})$, the infimum value over the period  $\bar{\Omega}_\theta$ is taken  at $\theta T+t_{1}$. Then according to Remark \ref{Remark3},  one can check that $d_0=\mathop{\inf}\limits_{t\in\bar{\Omega}_\theta}x(t)=\mathop{\inf}\limits_{t\in\bar{\Omega}_0}x(t)=x(t_{1})=\mathop{\min} \limits_{t \ge 0} x(t)$.  From $x(t)\in[0,1]$,  it follows that
 \begin{equation}\label{eq19}
\begin{aligned}
\mathop{\max} \limits_{t \ge 0} \left|x(t)-1\right|&=\mathop{\max} \limits_{t \ge 0} \{1-x(t)\}=1-\mathop{\min} \limits_{t \ge 0} x(t)=1-x(t_{1})\\
&=1-\frac{1}{1+\frac{1-x_0}{x_0}e^{-\alpha_{\lambda_1}t_{1}}}.
\end{aligned}
\end{equation}
 After calculation of $\mathop{\max} \limits_{t \ge 0} \left|x(t)-1\right|<\epsilon$, one derives that $1-x_0=\left|x_0-1\right|<1-\frac{1}{1+\frac{\epsilon}{1-\epsilon}e^{\alpha_{\lambda_1}t_{1}}}$. Thus, if $\mathop{\max} \limits_{t \ge 0} \left|x(t)-1\right|<\epsilon$, then $\left|x(t)-1\right|<\epsilon$.  This indicates that for any $\epsilon>0$, there exists a $\delta=1-\frac{1}{1+\frac{\epsilon}{1-\epsilon}e^{\alpha_{\lambda_1}t_{1}}}>0$ such that $\left|x(t)-1\right|<\epsilon, \forall \,\, t\geq 0$, whenever $\left|x_0-1\right|<\delta$.

For $d_{\theta}=x(\theta T+t_{j})$,  one knows that the infimum value over the period  $\bar{\Omega}_\theta$ is taken  at $\theta T+t_{j}$. Then according to Remark \ref{Remark3},
 one can check that $d_0=\mathop{\inf}\limits_{t\in\bar{\Omega}_\theta}x(t)=\mathop{\inf}\limits_{t\in\bar{\Omega}_0}x(t)=x(t_{j})=\mathop{\min} \limits_{t \ge 0} x(t)$.  It  follows from $x(t)\in[0,1]$ that
 \begin{equation}\label{eq20}
\begin{aligned}
\mathop{\max}  \limits_{t \ge 0}\left|x(t)-1\right|
&=\mathop{\max} \limits_{t \ge 0}\{1-x(t)\}=1-\mathop{\min} \limits_{t \ge 0}{x}(t)=1-x(t_{j})\\
&=1-\frac{1}{1+\frac{1-x_0}{x_0}e^{-\sum\limits_{l=1}^{j-1}(\alpha_{\lambda_l}-\alpha_{\lambda_{l+1}})t_{l}-\alpha_{\lambda_j}t_j}}\\
&=1-\frac{1}{1+\frac{1-x_0}{x_0}e^{-\sum\limits_{l=1}^{j}\alpha_{\lambda_l}(t_{l}-t_{l-1})}}.
\end{aligned}
\end{equation} After calculation of $\mathop{\max} \limits_{t \ge 0}\left|x(t)-1\right|<\epsilon$,  one can get that
$1-x_0=\left|x_0-1\right|<1-\frac{1}{1+\frac{\epsilon}{1-\epsilon}e^{\sum\limits_{l=1}^{j}\alpha_{\lambda_l}(t_{l}-t_{l-1})}}$. Therefore, if $\mathop{\max}  \limits_{t \ge 0} \left|x(t)-1\right|<\epsilon$, then $\left|x(t)-1\right|<\epsilon$.  This means that for any $\epsilon>0$, there exists a $\delta=1-\frac{1}{1+\frac{\epsilon}{1-\epsilon}e^{\sum\limits_{l=1}^{j}\alpha_{\lambda_l}(t_{l}-t_{l-1})}}>0$ such that $\left|x(t)-1\right|< \epsilon, \forall \,\, t\geq 0$, whenever $\left|x_0-1\right|<\delta$.

Thus, we have proven that the equilibrium point $x^*=1$ of \eqref{eq6}  is stable.
\end{proof}

\begin{theorem}\label{theorem3}
If \eqref{eq9} holds, the equilibrium point $x^*=1$  of the periodically switched system \eqref{eq6} is asymptotically stable.
\end{theorem}
\begin{proof}
In accordance with the stability analysis of $x^\ast=1$ in Theorem \ref{theorem2}, we further consider the asymptotic stability of \eqref{eq6}  from three cases, including $d_{\theta}=x(\theta T)$, $d_{\theta}=x(\theta T+t_{1})$, and $d_{\theta}=x(\theta T+t_{j})$ for all $j\in\mathcal{W}$ and $\theta\in \mathbb{N}$.
To prove that $\mathop{\lim}\limits_{t\rightarrow+\infty}x(t)=1$ holds,  for any $\epsilon>0$,
we should search for a $\hat{t}>0$ such that  $\left|x(t)-1\right|<\epsilon$ for all $t>\hat{t}$.

For $d_{\theta}=x(\theta T)$, under  \eqref{eq9}  one gets
\begin{equation}\label{eq21}
\begin{aligned}
 \mathop{\lim}\limits_{\theta\rightarrow+\infty}d_{\theta}&=\mathop{\lim}\limits_{\theta\rightarrow+\infty}\mathop{\inf} \limits_{t\in\bar{\Omega}_\theta}x(t)=\mathop{\lim}\limits_{\theta\rightarrow+\infty}x(\theta T)\\
 &=\mathop{\lim}\limits_{\theta\rightarrow+\infty} \frac{1}{1+\frac{1-x_0}{x_0}e^{-\Lambda_{1\theta}(\theta T)}}\\
 &=\mathop{\lim}\limits_{\theta\rightarrow+\infty} \frac{1}{1+\frac{1-x_0}{x_0}e^{-\theta\Big[\sum\limits_{i=1}^{m-1}(\alpha_{\lambda_i}-\alpha_{\lambda_{i+1}})t_{i}+\alpha_{\lambda_m}T\Big]}}\\
  &=\mathop{\lim}\limits_{\theta\rightarrow+\infty} \frac{1}{1+\frac{1-x_0}{x_0}e^{-\theta\sum\limits_{i=1}^{m} \alpha_{\lambda_i}(t_{i}-t_{i-1})}}=1.
\end{aligned}
\end{equation}

For $d_{\theta}=x(\theta T+t_{1})$,  under \eqref{eq9}  one gets
 \begin{equation}\label{eq22}
\begin{aligned}
 \mathop{\lim}\limits_{\theta\rightarrow+\infty}d_{\theta}&=\mathop{\lim}\limits_{\theta\rightarrow+\infty}\mathop{\inf} \limits_{t\in\bar{\Omega}_\theta}x(t)=\mathop{\lim}\limits_{\theta\rightarrow+\infty} x(\theta T+t_{1})\\
 &=\mathop{\lim}\limits_{\theta\rightarrow+\infty} \frac{1}{1+\frac{1-x_0}{x_0}e^{-\Lambda_{2\theta}(\theta T+t_{1})}}\\
 &=\mathop{\lim}\limits_{\theta\rightarrow+\infty} \frac{1}{1+\frac{1-x_0}{x_0}e^{-\theta\Big[\sum\limits_{i=1}^{m-1}(\alpha_{\lambda_i}-\alpha_{\lambda_{i+1}})t_{i}+\beta_{m}T\Big]-\alpha_{\lambda_1}t_{1}}}\\
  &=\mathop{\lim}\limits_{\theta\rightarrow+\infty} \frac{1}{1+\frac{1-x_0}{x_0}e^{-\theta\sum\limits_{i=1}^{m} \alpha_{\lambda_i}(t_{i}-t_{i-1})-\alpha_{\lambda_1}t_{1}}}=1.
\end{aligned}
\end{equation}

For $d_{\theta}=x(\theta T+t_{j})$, $j\in\mathcal{W}$,  under  \eqref{eq9}   one gets
 \begin{equation}\label{eq23}
\begin{aligned}
\mathop{\lim}\limits_{\theta\rightarrow+\infty}d_{\theta}&=\mathop{\lim}\limits_{\theta\rightarrow+\infty}\mathop{\inf} \limits_{t\in\bar{\Omega}_\theta}x(t)=\mathop{\lim}\limits_{\theta\rightarrow+\infty} x(\theta T+t_{j})\\
 &=\mathop{\lim}\limits_{\theta\rightarrow+\infty} \frac{1}{1+\frac{1-x_0}{x_0}e^{-\Lambda_{(j+1)\theta }(\theta T+t_{j})}}\\
 &=\mathop{\lim}\limits_{\theta\rightarrow+\infty} \frac{1}{1+\frac{1-x_0}{x_0}e^{-\theta\Big[\sum\limits_{i=1}^{m-1}(\alpha_{\lambda_i}-\alpha_{\lambda_{i+1}})t_{i}+\alpha_{\lambda_m}T\Big]-\beta_{j}t_{j}}}\\
  &=\mathop{\lim}\limits_{\theta\rightarrow+\infty} \frac{1}{1+\frac{1-x_0}{x_0}e^{-\theta\sum\limits_{i=1}^{m} \alpha_{\lambda_i}(t_{i}-t_{i-1})-\alpha_{\lambda_j}t_{j}}}=1.
\end{aligned}
\end{equation}
Consequently,  from \eqref{eq22}-\eqref{eq24}, it follows that for  $v\in\mathcal{V}$,
 $
\mathop{\lim}\limits_{\theta\rightarrow+\infty}d_{\theta}= \mathop{\lim}\limits_{\theta\rightarrow+\infty}x(\theta T+t_{v})=1.
$
One can see that  the limit of the sequence $\{d_{\theta}\}_{\theta\in \mathbb{N}}$ is 1. That is, $\forall \,\, \epsilon>0$, ${\exists} \,\, \Theta>0$, s.t. $\theta>\Theta \Rightarrow \left|d_{\theta}-1\right|=\left|x(\theta T+t_{v})-1\right|<\epsilon$. In addition, taking $\hat{t}=\Theta T +t_{v} \in\bar{\Omega}_\Theta$, then $\left|x(t)-1\right|<\epsilon$ for $t \ge \hat{t}$, which means $\mathop{\lim}\limits_{t\rightarrow+\infty}x(t)=1$.  Therefore, we have proven that the equilibrium point $x^*=1$ of \eqref{eq6}  is asymptotically stable.
\end{proof}



 Combining Theorems \ref{theorem1}-\ref{theorem3},  under \eqref{eq9} one can conclude that
 the periodically switched system \eqref{eq6} has two equilibrium points, $x^*=0$ and $x^*=1$. The equilibrium point
$x^*=0$ is unstable, but  $x^*=1$ is asymptotically stable no matter whether  the infimum value of \eqref{eq6} is taken on the boundary or inside each period $\bar{\Omega}_\theta$.
 For any given activation sequence $(\alpha_{\lambda_1}, \dots, \alpha_{\lambda_j}, \dots, \alpha_{\lambda_m})$,  Theorem~\ref{theorem3}  implies that under \eqref{eq9},  the periodically switched system \eqref{eq6} can evolve into the full cooperation state. Therefore, all periodically switched systems in the set $\Delta$ can reach the full cooperation state  when different strategy update rules are reasonably arranged  and periodically varying.\\


Moving to the alternative case of \eqref{eq11}, we now analyze the stability of the obtained equilibria  (i.e., $x^\ast=0$ and $x^\ast=1$)  in Theorem~\ref{theorem4}, and further discuss  the asymptotic stability of the equilibrium point $x^\ast=0$ in  Theorem \ref{theorem5}, respectively.
\begin{theorem}\label{theorem4}
If \eqref{eq11} holds,   the equilibrium point  $x^\ast=0$ is stable and $x^\ast=1$ is unstable.
\end{theorem}
\begin{proof}
Firstly, we prove that the equilibrium point $x^\ast=0$ of \eqref{eq6}  is stable from three cases, that is,  the supremum value of system $x(t)$ over the period  $\bar{\Omega}_\theta$ is taken at $\theta T$, $\theta T+t_{1}$, and $\theta T+t_{j}$, that is, $y_{\theta}=x(\theta T)$, $y_{\theta}=x(\theta T+t_{1})$, and $y_{\theta}=x(\theta T+t_{j})$ for $j\in\mathcal{W}$ and $\theta\in \mathbb{N}$.  Based on the definition of Lyapunov stability~\cite{Khalil2002}, we shall find $\delta>0$, such that  $\left|x(t)\right|<\epsilon$ holds for $t\ge 0$.

For $y_{\theta}=x(\theta T)$, the supremum value  of $x(t)$ over the period  $\bar{\Omega}_\theta$ is taken at $\theta T$. From  Remark \ref{Remark3}, for all $\theta\in \mathbb{N}$ the sequence $\{y_{\theta}\}_{\theta\in \mathbb{N}}=\{\mathop{\sup} \limits_{t\in\bar{\Omega}_\theta}x(t)\}_{\theta\in \mathbb{N}}$ is strictly monotonically decreasing, which implies that
 \begin{equation}\label{eq24}
\begin{aligned}
y_{0}=\mathop{\sup} \limits_{t\in\bar{\Omega}_0}\left|x(t)\right|=x_0=\mathop{\max}  \limits_{t \ge 0} \left|x(t)\right|.
\end{aligned}
 \end{equation}
Thus, if $\mathop{\max} \limits_{t \ge 0} \left|x(t)\right|=x_0<\epsilon$, then  $\left|x(t)\right|<\epsilon$. This means  that for any $\epsilon>0$, there exists a $\delta=\epsilon>0$ such that $\left|x(t)\right|< \epsilon$ for all  $t\geq 0$, whenever $\left|x_0\right|<\delta$.

For $y_{\theta}=x(\theta T+t_{1})$, the supremum value  over the period  $\bar{\Omega}_\theta$ is taken at $\theta T+t_{1}$. Then due to Remark \ref{Remark3},  one can check that
 \begin{equation}\label{eq25}
\begin{aligned}
y_{0}&=\mathop{\sup} \limits_{t\in\bar{\Omega}_0}\left|x(t)\right|=x(t_{1})=\mathop{\max}  \limits_{t \ge 0} \left|x(t)\right|\\
&=\frac{1}{1+\frac{1-x_0}{x_0}e^{-\alpha_{\lambda_1}t_{1}}}.
\end{aligned}
\end{equation}
 After calculation of $\max \limits_{t \ge 0} \left|x(t)\right|<\epsilon$, one can derive that $\left|x_0\right|<\frac{1}{1+\frac{1-\epsilon}{\epsilon}e^{\alpha_{\lambda_1}t_{1}}}$. Thus, if $\max \limits_{t \ge 0} \left|x(t)\right|<\epsilon$, then $\left|x(t)\right|<\epsilon$.  This indicates that for any $\epsilon>0$, there exists a $\delta=\frac{1}{1+\frac{1-\epsilon}{\epsilon}e^{\alpha_{\lambda_1}t_{1}}}>0$ such that $\left|x(t)\right|<\epsilon, \forall \,\, t\geq 0$, whenever $\left|x_0\right|<\delta$.

For $y_{\theta}=x(\theta T+t_{j})$,  one knows that the supremum value  over the period  $\bar{\Omega}_\theta$ is taken at $\theta T+t_{j}$. Then according to Remark \ref{Remark3},  one gets
 \begin{equation}\label{eq26}
\begin{aligned}
y_{0}&=\mathop{\sup} \limits_{t\in\bar{\Omega}_0}\left|x(t)\right|=x(t_{j})=\mathop{\max}  \limits_{t \ge 0} \left|x(t)\right|\\
&=\frac{1}{1+\frac{1-x_0}{x_0}e^{-\sum\limits_{l=1}^{j-1}(\alpha_{\lambda_l}-\alpha_{\lambda_{l+1}})t_{l}-\alpha_{\lambda_j}t_j}}\\
&=\frac{1}{1+\frac{1-x_0}{x_0}e^{-\sum\limits_{l=1}^{j}\alpha_{\lambda_l}(t_{l}-t_{l-1})}}.
\end{aligned}
\end{equation} After calculation of $\max \limits_{t \ge 0}\left|x(t)\right|<\epsilon$,  one can get that
$\left|x_0\right|<\frac{1}{1+\frac{1-\epsilon}{\epsilon}e^{\sum\limits_{l=1}^{j}\alpha_{\lambda_l}(t_{l}-t_{l-1})}}$. Thus, if $\max \limits_{t \ge 0} \left|x(t)-1\right|<\epsilon$, then $\left|x(t)-1\right|<\epsilon$.  This means that for any $\epsilon>0$, there exists a $\delta=\frac{1}{1+\frac{1-\epsilon}{\epsilon}e^{\sum\limits_{l=1}^{j}\alpha_{\lambda_l}(t_{l}-t_{l-1})}}>0$ such that $\left|x(t)\right|< \epsilon, \forall \,\, t\geq 0$, whenever $\left|x_0\right|<\delta$.
Hence, the equilibrium point $x^*=0$ of \eqref{eq6}  is stable.

On the other hand, we prove that the equilibrium point $x^\ast=1$ of \eqref{eq6}  is unstable.  According to Remark \ref{Remark2}, one can  get $y_{1}=\max\limits_{t \ge 0} x(t)$. Let $\epsilon_0=y_{1}\in\bar{\Omega}_1$,  then for all $\delta>0$, if $\left|x_0-1\right|<\delta$, we have $\left|x(t)-1\right|>\epsilon_0$ for $t \ge 0$. Therefore,  the equilibrium point $x^\ast=1$ is unstable.
\end{proof}

\begin{theorem}\label{theorem5}
If \eqref{eq11} holds, the equilibrium point $x^*=0$ of the periodically switched system \eqref{eq6}  is asymptotically stable.
\end{theorem}
\begin{proof}
Contingent on the stability analysis of $x^\ast=0$ in Theorem \ref{theorem4}, we can consider the asymptotic stability of \eqref{eq6}  from three cases, including $y_{\theta}=x(\theta T)$, $y_{\theta}=x(\theta T+t_{1})$, and $y_{\theta}=x(\theta T+t_{j})$ for $j\in\mathcal{W}$ and $\theta\in \mathbb{N}$.
To prove that $\mathop{\lim}\limits_{t\rightarrow+\infty}x(t)=0$ holds,  for any $\epsilon>0$,
we should search for a $\hat{t}>0$ such that  $\left|x(t)\right|<\epsilon$ for all $t>\hat{t}$.

For $y_{\theta}=x(\theta T)$, one can check under  \eqref{eq11}  that
\begin{equation}\label{eq27}
\begin{aligned}
 \mathop{\lim}\limits_{\theta\rightarrow+\infty}y_{\theta}&=\mathop{\lim}\limits_{\theta\rightarrow+\infty}\mathop{\sup} \limits_{t\in\bar{\Omega}_\theta}x(t)=\mathop{\lim}\limits_{\theta\rightarrow+\infty}x(\theta T)\\
 &=\mathop{\lim}\limits_{\theta\rightarrow+\infty} \frac{1}{1+\frac{1-x_0}{x_0}e^{-\Lambda_{1\theta}(\theta T)}}\\
 &=\mathop{\lim}\limits_{\theta\rightarrow+\infty} \frac{1}{1+\frac{1-x_0}{x_0}e^{-\theta\Big[\sum\limits_{i=1}^{m-1}(\alpha_{\lambda_i}-\alpha_{\lambda_{i+1}})t_{i}+\beta_{m}T\Big]}}\\
  &=\mathop{\lim}\limits_{\theta\rightarrow+\infty} \frac{1}{1+\frac{1-x_0}{x_0}e^{-\theta\sum\limits_{i=1}^{m} \alpha_{\lambda_i}(t_{i}-t_{i-1})}}=0.
\end{aligned}
\end{equation}

For $y_{\theta}=x(\theta T+t_{1})$,  one can check under  \eqref{eq11}  that
 \begin{equation}\label{eq28}
\begin{aligned}
 \mathop{\lim}\limits_{\theta\rightarrow+\infty}y_{\theta}&=\mathop{\lim}\limits_{\theta\rightarrow+\infty}\mathop{\sup} \limits_{t\in\bar{\Omega}_\theta}x(t)=\mathop{\lim}\limits_{\theta\rightarrow+\infty} x(\theta T+t_{1})\\
 &=\mathop{\lim}\limits_{\theta\rightarrow+\infty} \frac{1}{1+\frac{1-x_0}{x_0}e^{\Lambda_{2\theta}(\theta T+t_{1})}}\\
 &=\mathop{\lim}\limits_{\theta\rightarrow+\infty} \frac{1}{1+\frac{1-x_0}{x_0}e^{\theta\Big[\sum\limits_{i=1}^{m-1}(\alpha_{\lambda_i}-\alpha_{\lambda_{i+1}})t_{i}+\beta_{m}T\Big]-\alpha_{\lambda_1}t_{1}}}\\
  &=\mathop{\lim}\limits_{\theta\rightarrow+\infty} \frac{1}{1+\frac{1-x_0}{x_0}e^{-\theta\sum\limits_{i=1}^{m} \alpha_{\lambda_i}(t_{i}-t_{i-1})-\alpha_{\lambda_1}t_{1}}}=0.
\end{aligned}
\end{equation}

For $y_{\theta}=x(\theta T+t_{j})$, $j\in\mathcal{W}$,  one can check under  \eqref{eq11}  that
 \begin{equation}\label{eq29}
\begin{aligned}
\mathop{\lim}\limits_{\theta\rightarrow+\infty}y_{\theta}&=\mathop{\lim}\limits_{\theta\rightarrow+\infty}\mathop{\sup} \limits_{t\in\bar{\Omega}_\theta}x(t)=\mathop{\lim}\limits_{\theta\rightarrow+\infty} x(\theta T+t_{j})\\
 &=\mathop{\lim}\limits_{\theta\rightarrow+\infty} \frac{1}{1+\frac{1-x_0}{x_0}e^{-\Lambda_{(j+1)\theta }(\theta T+t_{j})}}\\
 &=\mathop{\lim}\limits_{\theta\rightarrow+\infty} \frac{1}{1+\frac{1-x_0}{x_0}e^{-\theta\Big[\sum\limits_{i=1}^{m-1}(\alpha_{\lambda_i}-\alpha_{\lambda_{i+1}})t_{i}+\beta_{m}T\Big]-\alpha_{\lambda_j}t_{j}}}\\
  &=\mathop{\lim}\limits_{\theta\rightarrow+\infty} \frac{1}{1+\frac{1-x_0}{x_0}e^{-\theta\sum\limits_{i=1}^{m}\alpha_{\lambda_i}(t_{i}-t_{i-1})-\alpha_{\lambda_j}t_{j}}}=0.
\end{aligned}
\end{equation}
Consequently, for  any $v\in\mathcal{V}$,
 $
\mathop{\lim}\limits_{\theta\rightarrow+\infty}y_{\theta}= \mathop{\lim}\limits_{\theta\rightarrow+\infty}x(\theta T+t_{v})=0.
$
One can see that the limit of the sequence  $\{y_{\theta}\}_{\theta\in \mathbb{N}}$ is 0. That is,  $\forall \,\, \epsilon>0$, ${\exists} \,\, \Theta>0$, s.t. $\theta>\Theta \Rightarrow \left|y_{\theta}\right|=\left|x(\theta T+t_{v})\right|<\epsilon$. Taking $\hat{t}=\Theta T +t_{v} \in\bar{\Omega}_\Theta$, then $\left|x(t)\right|<\epsilon$ for $t \ge 0$, which means $\mathop{\lim}\limits_{t\rightarrow+\infty}x(t)=0$. Therefore, we have proven that  the equilibrium point $x^*=0$ of \eqref{eq6}  is asymptotically stable.
\end{proof}



 By summarizing Theorems \ref{theorem1}, \ref{theorem4} and \ref{theorem5}, one can conclude  that
under \eqref{eq11}, the periodically switched system  \eqref{eq6} has two equilibrium points, $x^*=0$ and $x^*=1$. The equilibrium point
 $x^*=0$ is unstable, but  $x^*=1$ are asymptotically stable no matter whether the supremum value of \eqref{eq6} is taken on the boundary or inside each period $\bar{\Omega}_\theta$. From Theorem~\ref{theorem5}, we know that under \eqref{eq11},  all periodically switched systems  in the set $\Delta$ can evolve into the full defection state when different strategy update rules are periodically changing.





\begin{figure}[!t]
\centering
\includegraphics[width=3.5in]{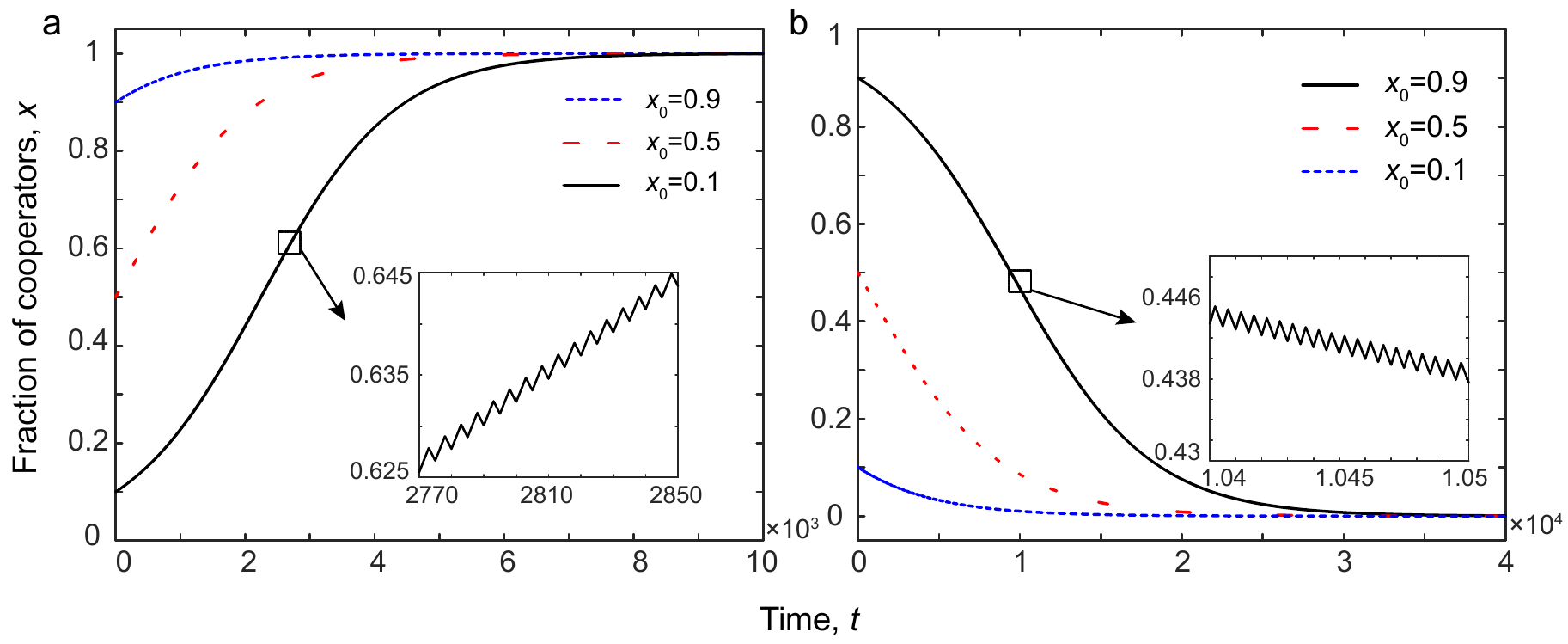}
\caption{Time evolution of the fraction of cooperators for  ``first IM updating, then PC updating'' at each period $\bar{\Omega}_\theta$ under different initial states. Three different initial states are $x_0=0.1$,  $x_0=0.5$, and $x_0=0.9$, respectively.
 The switching point is $t_{1}=3$ for panel $a$, and $t_{1}=2$ for panel $b$.
Other parameters are the same as those in Fig. \ref{fig2}.}
\label{fig3}
\end{figure}

\section{Numerical Results}
In this section, we provide a specific example to support the theoretical condition for the emergence of cooperative behavior we obtained. To this end, we study the simplest case, where individuals can periodically switch their strategy update rules between the PC and IM updating. Specifically, we first apply the pair approximation approach to describe how  the  fraction  of cooperators  on a regular network changes over time.  The resulting replicator equation for $\omega \ll 1$ limit under PC rule is given by
\begin{equation}\label{eq30}
\begin{aligned}
\frac{\textrm{d}x(t)}{\textrm{d}t}=-\frac{\omega k(k-2)c}{2(k-1)}x(t)(1-x(t)),
\end{aligned}
\end{equation}
while in the case of IM updating, it is
\begin{equation}\label{eq31}
\begin{aligned}
\frac{\textrm{d}x(t)}{\textrm{d}t}=\frac{\omega k^{2}(k-2)[b-(k+2)c]}{(k+1)^{2}(k-1)}x(t)(1-x(t)),
\end{aligned}
\end{equation}
where $b>(k+2)c$. Let $\alpha_1=-\frac{\omega k(k-2)c}{2(k-1)}$, and let $\alpha_2=\frac{\omega k^{2}(k-2)[b-(k+2)c]}{(k+1)^{2}(k-1)}$. The details of calculations are given in Appendix A.  For both replicator equations above, there exist two equilibria on the boundary of the parameter space, namely, $x^*=0$ and $x^*=1$. Under  the PC rule, the equilibrium  point $x^*=0$ is (asymptotically) stable and  $x^*=1$ is unstable. Under the IM rule and under  the condition that $b>(k+2)c$, the equilibrium  point  $x^*=0$  is unstable and  $x^*=1$  is (asymptotically) stable.

At each period $\bar{\Omega}_\theta$, the binary sequence set $\mathcal{A}$  is $
\mathcal{A}=\{(\alpha_{1},\alpha_{2}), (\alpha_{2}, \alpha_{1})\},
$
 and the periodically switched system set under the sequence set $\mathcal{B}$ is
$
\Delta=\{\Sigma|_{(\alpha_{\lambda_1}, \alpha_{\lambda_2})} |\,\,(\lambda_1, \lambda_2)\in\mathcal{B}\},
$
where $\mathcal{B}=\{(1,2),(2,1) \}$.

We consider ``first  PC updating, then IM updating'' over the period $\bar{\Omega}_\theta$, where the system equation \eqref{eq31}  with PC updating  represents subsystem 1, and \eqref{eq32}  with IM updating is subsystem 2.  Accordingly,  the periodically switched system in Eq.~\eqref{eq6}  can be rewritten as
\begin{equation}\label{eq32}
\frac{\textrm{d}x(t)}{\textrm{d}t}\bigg|_{(\alpha_{\lambda_1}, \alpha_{\lambda_2})}=\left\{
\begin{aligned}
\alpha_{\lambda_1}x(t)(1-x(t)),& &t\in\Omega_{1\theta},\\
\alpha_{\lambda_2}x(t)(1-x(t)),& &t\in\Omega_{2\theta},
\end{aligned}
\right.
\end{equation}
 where $(\alpha_{\lambda_1}, \alpha_{\lambda_2})=(\alpha_1, \alpha_2)$, $\Omega_{1\theta}=[ \theta T,  \theta T+t_{1})$, and $\Omega_{2\theta}=[ \theta T+t_{1},  (\theta+1) T)$ for all $\theta \in \mathbb{N}$. According to Theorem \ref{theorem1},  we yield that when $t_{1}\neq p_1$, there exist two equilibria, that is,   $x^*=1$ and $x^*=0$. From Theorems \ref{theorem3} and \ref{theorem5}, it follows that the equilibrium point $x^*=1$  is (asymptotically) stable if $t_{1}<p_1$, and  $x^*=0$  is (asymptotically) stable if $t_{1}>p_1$, where $p_1=\frac{\alpha_{\lambda_2}T}{\alpha_{\lambda_2}-\alpha_{\lambda_1}}=\frac{2k[b-(k+2)c]T}{2kb-(k^2+2k-1)c}$.
  Furthermore, we have also made numerical calculations to test our obtained  analytical results.  Figure \ref{fig2}  illustrates the fraction of cooperators $x(t)$ as a function of time $t$ for three initial states $x_0=0.1$, $x_0=0.5$, and $x_0=0.9$.  For each initial state, one observes that  when $t_{1}<p_1$ and   $x_0\in (0, 1)$, the system converges to the full cooperation state as shown in Fig.~\ref{fig2}$a$.  But when $t_{1}>p_1$ and  $x_0\in (0, 1)$, the system evolves to the full defection state as shown in Fig.~\ref{fig2}$b$.

Then, we consider an alternative case of ``first IM updating, then PC updating''  over the period $\bar{\Omega}_\theta$, where
the system equation \eqref{eq32}  with IM updating  is subsystem 1, and \eqref{eq31}  with PC updating represents subsystem 2. Accordingly,  the periodically switched system can be rewritten as
\begin{equation}\label{eq33}
\frac{\textrm{d}x(t)}{\textrm{d}t}\bigg|_{(\alpha_{\lambda_1}, \alpha_{\lambda_2})}=\left\{
\begin{aligned}
\alpha_{\lambda_1}x(t)(1-x(t)),& &t\in\Omega_{1\theta},\\
\alpha_{\lambda_2}x(t)(1-x(t)),& &t\in\Omega_{2\theta},
\end{aligned}
\right.
\end{equation}
 where $(\alpha_{\lambda_1}, \alpha_{\lambda_2})=(\alpha_2, \alpha_1)$,  $\Omega_{1\theta}=[ \theta T,  \theta T+t_{1})$, and $\Omega_{2\theta}=[ \theta T+t_{1},  (\theta+1) T)$ for all $\theta \in \mathbb{N}$. According to Theorem \ref{theorem1},  we yield that when $t_{1}\neq p_2$, there exist two equilibria, that is,   $x^*=1$ and $x^*=0$.
From Theorems \ref{theorem3} and \ref{theorem5}, it is checked that the equilibrium point  $x^*=1$  is (asymptotically) stable if $t_{1}>p_2$ and   $x_0\in (0, 1)$, and $x^*=0$ is (asymptotically) stable if $t_{1}<p_2$ and $x_0\in (0, 1)$,  where $p_2= \frac{\alpha_{\lambda_2}T}{\alpha_{\lambda_2}-\alpha_{\lambda_1}}=\frac{(k+1)^2cT}{2kb-(k^2+2k-1)c}$.
Accordingly, we validate our analytical results by means of numerical calculations as presented in Fig.~\ref{fig3}. As Fig.~\ref{fig3}$a$ illustrates, for $t_{1}>p_2$, the system converges to the full cooperation state. Otherwise, the system evolves to the full defection state as shown in Fig.~\ref{fig3}$b$.

\section{Discussion and Conclusion}

In this work, we systematically study the problem of cooperation on regular networks in the scenario of periodic switching of strategy update rules in a game-theoretical framework.  By using the approach of switched system theory, we derive the theoretical  condition for the promotion of cooperation. Under this condition, the periodically switched system with different strategy update rules can converge to the full cooperation state. In addition,
we  consider a concrete example of the PC and IM updating, and find that our numerical results verify our theoretical conclusions.  Our work provides an important insight into understanding the evolutionary dynamics of cooperation for theoretically modeling those individuals who use different update rules periodically.

We note that, however, our present work studies the evolutionary dynamics of cooperation under periodic switching of update rules on regular networks. Indeed, the regular network is a specific and simple form of networks. Thus, it is interesting to extend our work to general interaction networks. Moreover, our work focuses on the framework of the prisoner's dilemma game, which is a classic paradigm for studying the evolution of cooperation. There are other prototypical two-player dilemmas, such as snowdrift game~\cite{Doebeli05EL} and stag-hunt game~\cite{Skyrms04}, which are also typical paradigms for studying the problem of cooperation. Hence it is meaningful to study how the periodic switching of strategy update rules influences the evolutionary dynamics of cooperation in these games. Furthermore, apart from the simple pairwise-comparison and imitation rules, one can consider more sophisticated decision-making mechanisms, such as the multi-armed bandit problem~\cite{Huo17RSOS}. The periodically switching between exploration and exploitation in a decision-making process of the problem above can provide some insights into the emergence of cooperative behavior and its stability on complex networks, which needs further investigation.

In addition, the switching we considered is exogenous and independent of the system states. However, in some realistic scenarios, the switching rule may depend on the system states, and thus a promising extension of this work is to consider endogenous switching, where the switching rule to be used depends on the states, such as the frequency of one type. Furthermore, the periodic switching of update rules studied here can be regarded as a control scheme, which has some important effect on the behavior of complex systems; other forms of control schemes, such as under the influence of  technology, policy, market factors~\cite{Hu21Games}, and other small perturbations of system parameters, can be considered as a driven forces for evolutionary game dynamics, which are worthy of study in the future~\cite{Wang20PRSA}.




\appendices
 \section{Derivations of Equations (30) and (31)}
In this section,  we provide a detailed theoretical analysis to derive the governing dynamical equation for the PC and IM updating by using pair approximation approach, which are defined by Eqs.~(30)  and (31),  respectively.


\subsection*{A.1 \,\, PC Updating}

Based on the pair approximation approach, let $x_{C}$ and $x_{D}$ denote the proportion of $C$-players (i.e., cooperators) and $D$-players (i.e., defectors) in the whole population. Furthermore, we denote the proportion of $CC$, $CD$, $DC$, and $DD$ pairs
by $x_{CC}$, $x_{CD}$, $x_{DC}$, and $x_{DD}$, respectively. Finally let $x_{i|j}$  denote the conditional probability of finding an $i$-player given that the neighboring node is a $j$-player, where $i, j \in \{C, D\}$. Using these notations, we have that
$
x_{C}+x_{D}=1,
$
$
x_{C|i}+x_{D|i}=1,
$
$
x_{ij}=x_{i|j}x_{j},
$
and
$
x_{CD}=x_{DC}.
$
In addition, it can be checked that $x_{C}$ and $x_{C|C}$ can be used to characterize the evolutionary dynamics of a system since
$
 x_{D}=1-x_{C},
$
$
x_{D|C}= 1-x_{C|C},
$
$
x_{DC}= x_{CD}= x_{C}x_{D|C}= x_{C}(1-x_{C|C}),
$$
x_{C|D}=\frac{x_{CD}}{x_{D}}=\frac{x_{C}(1-x_{C|C})}{1-x_{C}},$
$x_{D|D}=1-x_{C|D}=\frac{1-2x_{C}+x_{C}x_{C|C}}{1-x_{C}},$
and
$x_{DD}=x_{D}x_{D|D}=1-2x_{C}+x_{C}x_{C|C}.$

For PC updating~\cite{Ohtsuki2006JTB},  at each time step we randomly choose a focal player from the  population to revise its strategy who has $l$ cooperators and $k-l$ defectors among its $k$ neighbors.   If the focal player adopts strategy $D$, then the fitness of the focal player is
\begin{equation}
\begin{aligned}
g_{F}^D&=1-\omega+\omega \pi_{F}^D=1-\omega+\omega lb,
\end{aligned}
 \label{A1}
\tag{A1}
\end{equation}
and the fitness of a $C$-neighbor is
\begin{equation}
\begin{aligned}
g_{C}^D&=1-\omega+\omega \pi_{C}^D\\
&=1-\omega+\omega\Big\{(k-1)x_{C|C}(b-c)-\Big[(k-1)x_{D|C}+1\Big]c\Big\},
\end{aligned}
 \label{A2}
\tag{A2}
\end{equation}
where $\pi_{F}^D$ and $\pi_{C}^D$ denote the payoffs of the focal player and its $C$-neighbor, respectively.

Since the focal player either keeps its current strategy or adopts the strategy of a neighbor with a probability that depends on the payoff difference, i.e., $\pi_{C}^D-\pi_{F}^D$, the probability  that the focal player adopts the strategy of a $C$-neighbor for $ \omega\rightarrow 0$ limit  is
\begin{equation}
\begin{aligned}
\Lambda=\frac{1}{1+e^{-\omega(\pi_{C}^D-\pi_{F}^D)}}=\frac{1}{2}+\omega\frac{\pi_{C}^D-\pi_{F}^D}{4}.
\end{aligned}
 \label{A3}
\tag{A3}
\end{equation}
Due to $g_{C}^D-g_{F}^D=\omega(\pi_{C}^D-\pi_{F}^D)$ for weak selection, one further yields $\Lambda=\frac{1}{2}+\frac{g_{C}^D-g_{F}^D}{4}$.

Therefore, $x_{C}$ increases by $\frac{1}{n}$ with probability
\begin{equation}
\begin{aligned}
 P\Big(\Delta x_{C}=\frac{1}{n}\Big)=x_{D}\sum_{l=0}^{k} \binom{k}{l} (x_{C|D})^{l}(x_{D|D})^{k-l} \frac{l}{k}\Lambda,
 \end{aligned}
  \label{A4}
\tag{A4}
\end{equation}
where $x_{D}\binom{k}{l} (x_{C|D})^{l}(x_{D|D})^{k-l}$ represents the probability that in the population, a defector having $l$ $C$-neighbors is randomly selected.

Consequently, the number of $CC$-pairs increases by $(k-1)x_{C|D}+1$ and $x_{CC}$ increases by $\frac{(k-1)x_{C|D}+1}{kn/2}$
with probability
\begin{equation}
\begin{aligned}
 P&\Big(\Delta x_{CC}=\frac{(k-1)x_{C|D}+1}{kn/2}\Big)=x_{D}\sum_{l=0}^{k} \binom{k}{l}\\
  &\quad (x_{C|D})^{l}(x_{D|D})^{k-l} \frac{l}{k}\Lambda.
\end{aligned}
  \label{A5}
\tag{A5}
\end{equation}

In addition, we consider another case where the randomly selected focal player adopts strategy $C$. The fitness of the focal player is
\begin{equation}
\begin{aligned}
g_{F}^C&=1-\omega+\omega \pi_{F}^C\\
&=1-\omega+\omega\Big[l(b-c) - (k-l) c\Big],
\end{aligned}
  \label{A6}
\tag{A6}
\end{equation}
and the fitness of a $D$-neighbor is
\begin{equation}
\begin{aligned}
g_{D}^C&=1-\omega+\omega \pi_{D}^C\\
&=1-\omega+\omega\Big[(k-1)x_{C|D}+1\Big] b,
\end{aligned}
  \label{A7}
\tag{A7}
\end{equation}
where $\pi_{F}^C$ and $\pi_{D}^C$ denote the payoffs of the focal player and its $D$-neighbor, respectively.
Then, the probability that the focal player adopts the strategy of a $D$-neighbor for $ \omega\rightarrow 0$ limit is defined by
\begin{equation}
\begin{aligned}
\Omega=\frac{1}{1+e^{-\omega(\pi_{D}^C-\pi_{F}^C)}}=\frac{1}{2}+\omega\frac{\pi_{D}^C-\pi_{F}^C}{4}=\frac{1}{2}+\frac{g_{D}^C-g_{F}^C}{4}.
\end{aligned}
  \label{A8}
\tag{A8}
\end{equation}
Therefore, $x_{C}$ decreases by $\frac{1}{n}$ with probability
\begin{equation}
\begin{aligned}
 P\Big(\Delta x_{C}=-\frac{1}{n}\Big)=x_{C}\sum_{l=0}^{k} \binom{k}{l}(x_{C|C})^{l}(x_{D|C})^{k-l} \frac{k-l}{k}\Omega,
 \end{aligned}
   \label{A9}
\tag{A9}
\end{equation}
where $x_{C} \binom{k}{l}(x_{C|C})^{l}(x_{D|C})^{k-l}$ represents the probability that in the population, a cooperator having $l$ $C$-neighbors  is randomly selected.

Furthermore, the number of $CC$-pairs decreases by $(k-1)x_{C|C}$, and hence $x_{CC}$ decreases by $\frac{2(k-1)x_{C|C}}{kn}$ with probability
\begin{equation}
\begin{aligned}
 P&\Big(\Delta x_{CC}=-\frac{2(k-1)x_{C|C}}{kn}\Big)&=x_{C}\sum_{l=0}^{k} \binom{k}{l}
 \\
  &\quad(x_{C|C})^{l}(x_{D|C})^{k-l} \frac{k-l}{k}\Omega.
\end{aligned}
   \label{A10}
\tag{A10}
\end{equation}
Suppose that one replacement event occurs in one unit of time $t$, and one can obtain the derivative of $x_{C}$ with respect to $t$ as
\begin{equation}
\begin{aligned}
& \frac{\textrm{d}x_{C}(t)}{\textrm{d}t}=\frac{E(\Delta x_{C})}{\Delta t}\\
&=\frac{\frac{1}{n}P(\Delta x_{C}=\frac{1}{n})-\frac{1}{n}P(\Delta x_{C}=-\frac{1}{n})}{\frac{1}{n}} \\
 &=\omega \Psi(x_{C}, x_{C\mid C}, \omega), \\
  &=\omega \bigg\{\frac{1}{2}x_{C}(1-x_{C|C})\Big[\frac{(k-1)b(x_{C\mid C}-x_{C})}{1-x_{C}}\\
   &-kc-b\Big]+o(\omega) \bigg\}.
 \end{aligned}
 \label{A11}
\tag{A11}
\end{equation}

 The derivative of $x_{CC}$ with respect to $t$ is
\begin{equation}
\begin{aligned}
 &\frac{\textrm{d}x_{CC}(t)}{\textrm{d}t}
=\frac{1}{\Delta t}E(\Delta x_{CC})\\
&=\frac{1}{\frac{1}{n}}\bigg\{\frac{(k-1)x_{C|D}+1}{kn/2} P\Big(\Delta x_{CC}=\frac{(k-1)x_{C|D}+1}{kn/2}\Big) \\
&-\frac{2(k-1)x_{C|C}}{kn} P\Big(\Delta x_{CC}=-\frac{2(k-1)x_{C|C}}{kn}\Big) \bigg\}\\
&=\frac{x_{C}(1-x_{C|C})}{k}\Big[1-(k-1)\frac{x_{C\mid C}-x_{C}}{1-x_{C}}\Big]+o(\omega).
\end{aligned}
 \label{A12}
\tag{A12}
\end{equation}

Due to $x_{C\mid C}=\frac{x_{CC}}{x_{C}}$, one can derive the time derivative of $x_{C\mid C}$ as
\begin{equation}
\begin{aligned}
\frac{\textrm{d}x_{C\mid C}(t)}{\textrm{d}t}&=\Phi(x_{C}, x_{C\mid C}, \omega)\\
&=\frac{1-x_{C|C}}{k}\Big[1-(k-1)\frac{x_{C\mid C}-x_{C}}{1-x_{C}}\Big]+o(\omega).
\end{aligned}
 \label{A13}
\tag{A13}
\end{equation}

Combining \eqref{A11}  and \eqref{A13}, we  have the following dynamical system
\begin{equation} \label{A14}
\tag{A14}
\left\{\begin{array}{lc}
  \frac{\textrm{d}x_{C}(t)}{\textrm{d}t}=\omega \Psi(x_{C}, x_{C\mid C}, \omega),\\
 \frac{\textrm{d}x_{C\mid C}(t)}{\textrm{d}t}= \Phi(x_{C}, x_{C\mid C}, \omega).
\end{array}\right.
\end{equation}

System \eqref{A14}  can be rewritten with a change in time scale as
\begin{equation}
\label{A15}
\tag{A15}
\left\{\begin{array}{lc}
  \frac{\textrm{d}x_{C}(\tau)}{\textrm{d}\tau}=\Psi(x_{C}, x_{C\mid C}, \omega),\\
 \omega\frac{\textrm{d}x_{C\mid C}(\tau)}{\textrm{d}\tau}=\Phi(x_{C}, x_{C\mid C}, \omega),
\end{array}\right.
\end{equation}
where $ \tau=\omega t$. We refer to the time scale given by $\tau$ as slow, whereas the
time scale for $t$ is fast.
Further, as long as $\omega\neq0$, the two systems are equivalent
and are referred to as singular perturbation when $0<\omega\ll1$. Letting $\omega\rightarrow0$ in \eqref{A15}, we obtain the
system
\begin{equation}\label{A16}
\tag{A16}
\left\{\begin{array}{lc}
  \frac{\textrm{d}x_{C}(\tau)}{\textrm{d}\tau}=\Psi(x_{C}, x_{C\mid C}, 0),\\
 0=\Phi(x_{C}, x_{C\mid C}, 0),
\end{array}\right.
\end{equation}
which is called the reduced model.
One thinks of the condition $\Phi(x_{C}, x_{C\mid C}, 0)=0$ as determining a set on which the flow
is given by $\frac{\textrm{d}x_{C}(\tau)}{\textrm{d}\tau}=\Psi(x_{C}, x_{C\mid C}, 0)$.
For $\omega=0$, the set $\mathcal{V}=\{(x_{C}, x_{C\mid C})\mid\Phi(x_{C}, x_{C\mid C}, 0)=0\}$ consists of two  subsets
\begin{equation}
\label{A17}
\tag{A17}
\begin{aligned}
\mathcal{M}_{0}^{0}=\Big\{(x_{C}, x_{C\mid C})\big| x_{C\mid C}=1\Big\}=\Big\{(1, 1)\Big\},
\end{aligned}
\end{equation}
and
\begin{equation}
\label{A18}
\tag{A18}
\begin{aligned}
\mathcal{M}_{0}^{1}=\Big\{(x_{C}, x_{C\mid C})\big| x_{C|C}=\frac{1}{k-1}+\frac{k-2}{k-1}x_{C}\Big\}.
\end{aligned}
\end{equation}
It is worth noting that $\mathcal{M}_{0}^{0}$ is $\{(1, 1)\}$ since $x_{CC}=x_{C}x_{C|C}$, that is, if $x_{C|C}=1$, then $x_{C}=1$.
We prove this by contradiction. Suppose $x_{C}\neq1$, then there exists at least a $D$-player in a population. Under Assumption \ref{assumption1}, the network is connected, and thus at least a $C$-player is linked to
a $D$-player, which leads to $x_{D|C}>0$.  In addition, due to $x_{D|C}+ x_{C|C}=1$,  this implies that $x_{C|C}<1$,
contradicting  $x_{C|C}=1$. Therefore, we verify the above conclusion.

From Fenichel's Second Theorem~\cite{Jones1995DS}, only $\mathcal{M}_{0}^{1}$ is a normally hyperbolic compact manifold with a boundary. Then, for  $\varepsilon>0$ sufficiently small, there exists a local stable manifold $\mathcal{M}_{\varepsilon}^{1}$. This manifold lies within  $\mathcal{O}(\varepsilon)$ of $\mathcal{M}_{0}^{1}$ and is diffeomorphic to a stable manifold $\mathcal{M}_{0}^{1}$. Moreover, it is smooth and locally invariant under the flow of the system \eqref{A15}.  Therefore, by substituting \eqref{A18} into \eqref{A15},  one yields the reduced model
that can characterize the dynamical equation of \eqref{A15} for $\omega\rightarrow0$ limit, given by
\begin{equation}
\label{A19}
\tag{A19}
\begin{aligned}
\frac{\textrm{d}x_C(\tau)}{\textrm{d}\tau}=-\frac{k(k-2)c}{2(k-1)}x_{C}(1-x_{C}).
 \end{aligned}
\end{equation}
Let $t=\frac{\tau}{\omega}$ and $x_{C}=x$, and Eq.~\eqref{A19} can be rewritten as
\begin{equation}
\label{A20}
\tag{A20}
\begin{aligned}
\frac{\textrm{d}x(t)}{\textrm{d}t}=-\frac{\omega k(k-2)c}{2(k-1)}x(1-x),
 \end{aligned}
\end{equation}
 which has two equilibria $x^*=0$ and $x^*=1$.
Define the function $F(x)$ as
\begin{equation}
\label{A21}
\tag{A21}
\begin{aligned}
F(x)=-\frac{\omega k(k-2)c}{2(k-1)}x(1-x).
\end{aligned}
\end{equation}
The derivative of $F(x)$ with respect to $x$ is
\begin{equation}
\label{A22}
\tag{A22}
\begin{aligned}
\frac{\textrm{d}F(x)}{\textrm{d}x}=-\frac{\omega k(k-2)c}{2(k-1)}(1-2x),
\end{aligned}
\end{equation}
and  then at these two equilibria one gets that $\frac{\textrm{d}F(x)}{\textrm{d}x}\big|_{x^*=1}=\frac{\omega k(k-2)c}{2(k-1)}$ and $\frac{\textrm{d}F(x)}{\textrm{d}x}\big|_{x^*=0}=-\frac{\omega k(k-2)c}{2(k-1)}$. Under Assumption 2,  one yields that $\frac{\omega k(k-2)c}{2(k-1)}$ is positive. Therefore, when the initial state $x_0\in (0, 1)$, then   $\frac{\textrm{d}F(x)}{\textrm{d}x}\big|_{x^*=1}>0$  and $\frac{\textrm{d}F(x)}{\textrm{d}x}\big|_{x^*=0}<0$ for all $t\geq0$. This implies that when the initial state  $x_0\in (0, 1)$, the equilibrium point $x^*=1$ is unstable and $x^*=0$ is (asymptotically) stable, that is, cooperation can never emerge as observed in previous work \cite{Ohtsuki2006JTB}.

\subsection*{A.2 \,\, IM Updating}
For IM updating~\cite{Ohtsuki06Nature, Ohtsuki2006JTB}, at each time step a player of the population is randomly selected as the focal player and update its strategy who always has $l$ cooperators and $k-l$ defectors among its $k$ neighbors. If the focal player adopts strategy $D$, then the fitness of the focal player is
\begin{equation}
\label{A23}
\tag{A23}
\begin{aligned}
\bar{g}_{F}^D=g_{F}^D,
\end{aligned}
\end{equation}
the fitness of a $C$-neighbor is
\begin{equation}
\begin{aligned}
\bar{g}_{C}^D=g_{C}^D,
\end{aligned}
\label{A24}
\tag{A24}
\end{equation}
and the fitness of a $D$-neighbor is
\begin{equation}
\label{A25}
\tag{A25}
\begin{aligned}
\bar{g}_{D}^D=1-\omega+\omega\big[(k-1)x_{C\mid D}b\big].
\end{aligned}
\end{equation}

Since the focal player can either stay with its own strategy or imitate one of its neighbors' strategies with probability proportional to their fitness, the probability that the focal player adopts the strategy $C$ is given by
\begin{equation}
\label{A26}
\tag{A26}
\begin{aligned}
\Theta=\frac{l\bar{g}_{C}^D}{l\bar{g}_{C}^D+(k-l)\bar{g}_{D}^D+\bar{g}_{F}^D}.
\end{aligned}
\end{equation}
Therefore, $x_{C}$ increases by $\frac{1}{n}$ with probability
\begin{equation}
\label{A27}
\tag{A27}
\begin{aligned}
 P\Big(\Delta x_{C}=\frac{1}{n}\Big)=x_{D}\sum_{l=0}^{k}\binom{k}{l} (x_{C|D})^{l}(x_{D|D})^{k-l}\Theta.
 \end{aligned}
\end{equation}
Consequently, the number of $CC$-pairs increases by $l$ and hence $x_{CC}$ increases by
$\frac{2l}{kn}$ with probability
\begin{equation}
\label{A28}
\tag{A28}
\begin{aligned}
 P\Big(\Delta x_{CC}=\frac{2l}{kn}\Big)=x_{D}\binom{k}{l}(x_{C|D})^{l}(x_{D|D})^{k-l}\Theta.
 \end{aligned}
\end{equation}

In the alternative case, the randomly selected focal player adopts strategy $C$.  The fitness of the focal player is
\begin{equation}
\label{A29}
\tag{A29}
\begin{aligned}
\bar{g}_{F}^C=g_{F}^C,
\end{aligned}
\end{equation}
the fitness of a $D$-neighbor is
\begin{equation}
\label{A30}
\tag{A30}
\begin{aligned}
\bar{g}_{D}^C=g_{D}^C,
\end{aligned}
\end{equation}
and the fitness of a $C$-neighbor is
\begin{equation}
\label{A31}
\tag{A31}
\begin{aligned}
\bar{g}_{C}^C=1-\omega+ \omega\Big\{\big[(k-1)x_{C|C}+1\big](b-c)-(k-1)x_{D|C}c\Big\}.
\end{aligned}
\end{equation}
The probability that the focal player adopts the strategy $D$ is
\begin{equation}
\label{A32}
\tag{A32}
\begin{aligned}
\Upsilon=\frac{(k-l) \bar{g}_{D}^C}{l \bar{g}_{C}^C+(k-l) \bar{g}_{D}^C+\bar{g}_{F}^C}.
\end{aligned}
\end{equation}
Thus, $x_{C}$ decreases by $\frac{1}{n}$ with probability
\begin{equation}
\label{A33}
\tag{A33}
\begin{aligned}
 P\Big(\Delta x_{C}=-\frac{1}{n}\Big)= x_{C}\sum_{l=0}^{k} \binom{k}{l} (x_{C|C})^{l}(x_{D|C})^{k-l}\Upsilon.
 \end{aligned}
\end{equation}
Therefore, the number of $CC$-pairs decreases by $l$ and hence $x_{CC}$ decreases by $\frac{2 l}{kn}$   with probability
\begin{equation}
\label{A34}
\tag{A34}
\begin{aligned}
  P\Big(\Delta x_{CC}=-\frac{2 l}{kn}\Big)=x_{C}\binom{k}{l} (x_{C|C})^{l}(x_{D|C})^{k-l }\Upsilon.
  \end{aligned}
\end{equation}
From  these calculations, the time derivative of
$x_{C}$ is given by
\begin{equation}
\label{A35}
\tag{A35}
\begin{aligned}
 &\frac{\textrm{d}x_{C}(t)}{\textrm{d}t}=\frac{E(\Delta x_{C})}{\Delta t}\\
 &=\frac{\frac{1}{n}P(\Delta x_{C}=\frac{1}{n})-\frac{1}{n}P(\Delta x_{C}=-\frac{1}{n})}{\frac{1}{n}} \\
 &=\omega\bar{\Psi}(x_{C}, x_{C\mid C}, \omega)
 \\
 &=\omega\bigg\{\frac{k x_{C}(1-x_{C|C})}{(k+1)^{2}}\Big\{-2(kc+b)\\
 &+\frac{(k-1)b(x_{C\mid C}-x_{C})}{1-x_{C}}\Big[2+\frac{(k-1)(1-2x_{C}+x_{C\mid C})}{1-x_{C}}\\
 &-\frac{k(k-1)c(1-2x_{C}+x_{C\mid C})}{1-x_{C}}\Big]\Big\}+o(\omega) \bigg\}.
 \end{aligned}
\end{equation}

Accordingly, the time derivative of $x_{CC}$ is given by
 \begin{equation}
 \label{A36}
\tag{A36}
\begin{aligned}
&\frac{\textrm{d}x_{CC}(t)}{\textrm{d}t}=\frac{E(\Delta x_{CC})}{\Delta t}\\
 &=\frac{\sum_{l=0}^k \frac{2l}{kn}\big[P(\Delta x_{CC}=\frac{2l}{kn})- P(\Delta x_{CC}=-\frac{2l}{kn})\big]}{\frac{1}{n}}\\
&=\frac{2x_{C}(1-x_{C|C})}{k+1}\Big[1-(k-1)\frac{x_{C\mid C}-x_{C}}{1-x_{C}}\Big]+o(\omega).
\end{aligned}
\end{equation}
Furthermore, we have
\begin{equation}
 \label{A37}
\tag{A37}
\begin{aligned}
\frac{\textrm{d}x_{C\mid C}(t)}{\textrm{d}t}
&=\bar{\Phi}(x_{C}, x_{C\mid C}, \omega)\\
&=\frac{2(1-x_{C|C})}{k+1}\Big[1-(k-1)\frac{x_{C\mid C}-x_{C}}{1-x_{C}}\Big]+o(\omega).
\end{aligned}
\end{equation}

Combining \eqref{A35} and \eqref{A37}, we have the following dynamical system
\begin{equation}
 \label{A38}
\tag{A38}
 \left\{\begin{array}{lc}
 \frac{\textrm{d}x_{C}}{\textrm{d}t}= \omega\bar{\Psi}(x_{C}, x_{C\mid C}, \omega),\\
 \frac{\textrm{d}x_{C|C}}{\textrm{d}t}=\bar{\Phi}(x_{C}, x_{C\mid C}, \omega).
\end{array}\right.
\end{equation}

Similar to the PC rule, system \eqref{A38}  can be rewritten with a change in time scale as
\begin{equation}
 \label{A39}
\tag{A39}
\left\{\begin{array}{lc}
  \frac{\textrm{d}x_{C}(\tau)}{\textrm{d}\tau}=\bar{\Psi}(x_{C}, x_{C\mid C}, \omega),\\
 \omega\frac{\textrm{d}x_{C\mid C}(\tau)}{\textrm{d}\tau}=\bar{\Phi}(x_{C}, x_{C\mid C}, \omega),
\end{array}\right.
\end{equation}
where $\tau=\omega t$.  Letting $\omega\rightarrow0$ in \eqref{A39},  we obtain the reduced model
\begin{equation}
 \label{A40}
\tag{A40}
\left\{\begin{array}{lc}
  \frac{\textrm{d}x_{C}(\tau)}{\textrm{d}\tau}=\bar{\Psi}(x_{C}, x_{C\mid C}, 0),\\
 0=\bar{\Phi}(x_{C}, x_{C\mid C}, 0).
\end{array}\right.
\end{equation}
One thinks of the condition $\bar{\Phi}(x_{C}, x_{C\mid C}, 0)=0$ as determining a set on which the flow
is given by $\frac{\textrm{d}x_{C}(\tau)}{\textrm{d}\tau}=\bar{\Psi}(x_{C}, x_{C\mid C}, 0)$.
For $\omega=0$, the set $\mathcal{V}=\big\{(x_{C}, x_{C\mid C})\big|\bar{\Phi}(x_{C}, x_{C\mid C}, 0)=0\big\}$ consists of two  subsets
\begin{equation}
 \label{A41}
\tag{A41}
\begin{aligned}
\mathcal{M}_{0}^{0}=\Big\{(x_{C}, x_{C\mid C})\mid x_{C\mid C}=1\Big\}=\Big\{(1, 1)\Big\},
\end{aligned}
\end{equation}
and
\begin{equation}
 \label{A42}
\tag{A42}
\begin{aligned}
\mathcal{M}_{0}^{1}=\Big\{(x_{C}, x_{C\mid C})\big| x_{C|C}=\frac{1}{k-1}+\frac{k-2}{k-1}x_{C}\Big\}.
\end{aligned}
\end{equation}

From Fenichel's Second Theorem~\cite{Jones1995DS}, only $\mathcal{M}_{0}^{1}$ is a normally hyperbolic compact manifold with a boundary. Then, for  $\varepsilon>0$ sufficiently small, there exists a local stable manifold $\mathcal{M}_{\varepsilon}^{1}$. This manifold lies within  $\mathcal{O}(\varepsilon)$ of $\mathcal{M}_{0}^{1}$ and is diffeomorphic to a stable manifold $\mathcal{M}_{0}^{1}$. Moreover, it is smooth and locally invariant under the flow of the system \eqref{A39}.  Therefore, by substituting \eqref{A42} into the system \eqref{A39},  one yields the reduced model
that can characterize the dynamical equation of \eqref{A39}  for $\omega\rightarrow0$ limit, given by
\begin{equation}
 \label{A43}
\tag{A43}
\begin{aligned}
\frac{\textrm{d}x_{C}(\tau)}{\textrm{d}\tau} =\frac{k^{2}(k-2)[b-(k+2)c]}{(k+1)^{2}(k-1)}x_{C}(1-x_{C}).
 \end{aligned}
\end{equation}
Let $t=\frac{\tau}{\omega}$ and $x_{C}=x$, and Eq.~\eqref{A43} becomes
\begin{equation}
 \label{A44}
\tag{A44}
\begin{aligned}
\frac{\textrm{d}x(t)}{\textrm{d}t} =\frac{\omega k^{2}(k-2)[b-(k+2)c]}{(k+1)^{2}(k-1)}x(1-x),
 \end{aligned}
\end{equation}
 which has two fixed points $x^*=0$ and $x^*=1$. Define the function $\bar{F}(x)$ as
\begin{equation}
 \label{A45}
\tag{A45}
\begin{aligned}
\bar{F}(x)=\frac{\omega k^{2}(k-2)[b-(k+2)c]}{(k+1)^{2}(k-1)}x(1-x),
\end{aligned}
\end{equation}
and its derivative with respect to $x$ is
\begin{equation}
 \label{A46}
\tag{A46}
\begin{aligned}
\frac{\textrm{d}\bar{F}(x)}{\textrm{d}x}=\frac{\omega k^{2}(k-2)[b-(k+2)c]}{(k+1)^{2}(k-1)}(1-2x),
\end{aligned}
\end{equation}
and one gets that $\frac{\textrm{d}\bar{F}(x)}{\textrm{d}x}\big|_{x^*=1}=-\frac{\omega k^{2}(k-2)[b-(k+2)c]}{(k+1)^{2}(k-1)}$ and $\frac{\textrm{d}\bar{F}(x)}{\textrm{d}x}\big|_{x^*=0}=\frac{\omega k^{2}(k-2)[b-(k+2)c]}{(k+1)^{2}(k-1)}$. Hence,  if $b/c>k+2$,  one yields that $\frac{\omega k^{2}(k-2)[b-(k+2)c]}{(k+1)^{2}(k-1)}$ is positive. Therefore, under the condition that $b/c>k+2$, if the initial state $x_0\in (0, 1)$, then $\frac{\textrm{d}\bar{F}(x)}{\textrm{d}x}\big|_{x^*=1}<0$ and $\frac{\textrm{d}\bar{F}(x)}{\textrm{d}x}\big|_{x^*=0}>0$ for all $t\geq0$. This means that  the equilibrium point $x^*=1$ is (asymptotically)  stable and $x^*=0$ is unstable. Thus, we obtain the condition $b/c>k+2$ for the evolution of cooperation under the IM update rule as previously obtained in~\cite{Ohtsuki06Nature, Ohtsuki2006JTB}.

\ifCLASSOPTIONcaptionsoff
  \newpage
\fi


\begin{thebibliography}{1}

\bibitem{Hauert2010MIT}
S.~Hauert, S.~Mitri, L.~Keller, and D. ~Floreano,  \emph{Evolving Cooperation: From Biology to Engineering. The Horizons of Evolutionary Robotics}, MIT Press, 2010.

\bibitem{Perc2017PR}
M.~Perc, J.~J.~Jordan, D.~G.~Rand, Z.~Wang, S.~Boccaletti, and A.~Szolnoki,  \emph{Statistical physics of human cooperation}, Phys. Rep., vol. 687, pp. 1-51, 2017.


\bibitem{Vasconcelos2013NCC}
V.~V.~Vasconcelos, F.~C.~Santos, and J.~M.~Pacheco,  \emph{A bottom-up institutional approach to cooperative governance of risky commons},  Nat. Clim. Change, vol.  3, pp. 797-801, 2013.



\bibitem{Vincent 2005CUP}
T.~L.~Vincent, and J.~S.~Brown,  \emph{Evolutionary game theory, natural selection, and Darwinian dynamics}, Cambridge University Press, 2005.

\bibitem{Riehl22018ARC}
J.~Riehl, P.~Ramazi, and  M.~Cao,  \emph{A survey on the analysis and control of evolutionary matrix games},  Annual Reviews in Control, vol.  45, pp. 87--106, 2018.


\bibitem{Su2022SA}
Q.~Su,  A.~McAvoy, and J.~B.~Plotkin,  \emph{Evolution of cooperation with contextualized behavior},  Sci. Adv., vol.  8, pp. eabm6066, 2022.




\bibitem{Hofbauer1998CUP}
J.~Hofbauer, and K.~Sigmund,  \emph{Evolutionary games and population dynamics}, Cambridge University Press, 1998.



\bibitem{Vasconcelos2015M3AS}
V.~V.~Vasconcelos, F.~C.~Santos, and J.~M.~Pacheco,  \emph{Cooperation dynamics of polycentric climate governance}, Math. Models Methods Appl. Sci., vol. 25, pp. 2503-2517, 2015.




\bibitem{Su2022NHB}
Q.~Su,  A.~McAvoy, Y.~Mori, and J.~B.~Plotkin,  \emph{Evolution of prosocial behaviours in multilayer populations},  Nat. Hum. Behav., vol.  6, pp. 338-348, 2022.


  \bibitem{Nowak1992nature}
M.~ A.~Nowak, and R.~M.~May,  \emph{Evolutionary games and spatial chaos}, Nature, vol. 359, pp. 826-829, 1992.

  \bibitem{Perc2008PRE}
M.~Perc, and  A.~Szolnoki,  \emph{Social diversity and promotion of cooperation in the spatial prisoner's dilemma game}, Phys. Rev. E, vol. 77, pp. 011904, 2008.

  \bibitem{Perc2006NJP}
M.~Perc, and  A.~Szolnoki,  \emph{Coherence resonance in a spatial prisoner's dilemma game}, New J. of Phys., vol. 8, pp. 22, 2006.

\bibitem{Amaral2018PRE}
M. A. Amaral, and M. A. Javarone,  \emph{Heterogeneous update mechanisms in evolutionary games: mixing innovative and imitative dynamics}, Phys. Rev. E, vol. 97, pp. 042305, 2018.


 \bibitem{Blume1995BEB}
L. E. Blume,  \emph{The statistical mechanics of best-response strategy revision}, Games Econ. Behav., vol. 11, pp. 111--145, 1995.

\bibitem{Govaert2021TCNS}
A. Govaert, P. Ramazi, and M. Cao,  \emph{Rationality, Imitation, and Rational Imitation in Spatial Public Goods Games}, IEEE Trans. Control. Netw. Syst., vol. 8, pp. 1324-1335, 2021.



\bibitem{Como2020TCNS}
G.~Como, F.~Fagnani, and L.~Zino,  \emph{ Imitation dynamics in population games on community networks}, IEEE Trans. Control. Netw. Syst., vol. 8, pp. 65-76, 2020.


\bibitem{Allen2017nature}
B.~Allen, G.~Lippner, Y.~T.~Chen, B.~Fotouhi, N.~Momeni,  S.~T.~Yau, and  M.~A.~Nowak,  \emph{Evolutionary dynamics on any population structure}, Nature, vol. 544, pp.  227-230, 2017.

\bibitem{Banez2021WN}
R.~A.~Banez, L.~Li, C.~Yang, and Z.~Han,  \emph{Multiple-Population Mean Field Game for Social Networks}, Mean Field Game and its Applications in Wireless Networks. Wireless Networks. Springer, 2021: 113-145.


\bibitem{Barreiro16TSMC}
J.~Barreiro-Gomez, G.~Obando, and N.~Quijano,  \emph{Distributed population dynamics: Optimization and control applications}, IEEE Trans. Syst. Man Cybern., vol. 47(2), pp. 304-314, 2016.


\bibitem{Govaert22CSL}
A.~Govaert, L.~Zino, and E.~Tegling,  \emph{Population games on dynamic community networks}, IEEE Contr. Syst. Lett., vol. 6, pp. 2695-2700, 2022.

\bibitem{Ohtsuki06Nature}
H. Ohtsuki, C. Hauert,  E. Lieberman,  and M.~A. Nowak,  \emph{A simple rule for the evolution of cooperation on graphs and social networks}, Nature, vol. 441, pp. 502-505, 2006.

\bibitem{Szab07PR}
G. Szab\'{o},  and G. F\'{a}th,  \emph{Evolutionary games on graphs}, Phys. Rep., vol. 446, pp. 97-216, 2007.


 \bibitem{SzaboPRE98}
G. Szab\'{o}, and C. T\H{o}ke,  \emph{Evolutionary prisoner's dilemma game on a square lattice}, Phys. Rev. E, vol. 58, pp. 69, 1998.


  \bibitem{Ramazi2022Auto}
P. Ramazi, J. Riehl, and M. Cao,  \emph{The lower convergence tendency of imitators compared to best responders}, Automatica, vol. 139, pp. 110185, 2022.

\bibitem{Fu2011PRSB}
F.~Fu,  D.~I.~Rosenbloom, L.~Wang, and M.~A.~Nowak,  \emph{Imitation dynamics of vaccination behaviour on social networks}, Proc. Royal Soc. B, vol. 278, pp. 42-49, 2011.


 \bibitem{chen2008PRE}
X. Chen, and L. Wang,  \emph{Promotion of cooperation induced by appropriate payoff aspirations in a small-world networked game}, Phys. Rev. E, vol. 77, pp. 017103, 2008.




\bibitem{Ohtsuki2006JTB}
H. Ohtsuki, and M.~A. Nowak,  \emph{The replicator equation on graphs}, J. Theor. Biol., vol. 243, pp. 86-97, 2006.



 \bibitem{Shi2020TNSE}
L. Shi, C.  Shen, Q. Shi, Z. Wang, J. Zhao, X.  Li, and S. Boccaletti,  \emph{Recovering network structures based on evolutionary game dynamics via secure dimensional reduction}, IEEE Trans. Netw. Sci. Eng., vol. 7, pp. 2027-2036, 2020.


\bibitem{Zhang2019IEEETCS}
J.~Zhang, and M.~Cao,  \emph{Strategy competition dynamics of multi-agent systems in the framework of evolutionary game theory},  IEEE Trans. Circuits Syst. II: Express Briefs, vol.  67, pp. 152--156, 2019.

\bibitem{Hu2021TNSE}
Z. Hu, X. Li, J. Wang, C. Xia, Z. Wang, and M. Perc,  \emph{Adaptive reputation promotes trust in social networks},  IEEE Trans. Netw. Sci. Eng., vol. 8, pp. 3087¨C3098, 2021.



\bibitem{Fleming1986JM}
T. H. Fleming, and E. R. Heithaus,  \emph{Seasonal foraging behavior of the frugivorous bat Carollia perspicillata}, J. Mammal., vol. 67, pp. 660-671, 1986.

\bibitem{Hutto1981Auk}
R. L. Hutto,  \emph{Seasonal variation in the foraging behavior of some migratory western wood warblers}, Auk, vol. 98, pp. 765-777, 1981.

\bibitem{Conner11981Auk}
R. N.  Conner,  \emph{Seasonal changes in woodpecker foraging patterns}, Auk, vol. 98, pp. 562-570, 1981.





\bibitem{Clark1980book}
C. W. Clark,  \emph{Restricted access to common-property fishery resources: a game-theoretic analysis}, Dynamic optimization and mathematical economics, Springer, Boston, 1980.

\bibitem{Szolnoki2013SR}
A.~Szolnoki, and  M.~Perc,  \emph{Seasonal payoff variations and the evolution of cooperation in social dilemmas},  Sci. Rep., vol.  9, pp. 12575, 2013.



 \bibitem{Karlen1994AA}
D. L. Karlen, G. E. Varvel, D. G. Bullock,  and  R. M. Cruse,  \emph{ Crop rotations for the 21st century}, Adv. Agron., vol. 53, pp. 1-45, 1994.





 \bibitem{Wang2020TNSE}
 W. Wang and X. Li,  \emph{Temporal stable community in time-varying networks}, IEEE Trans. Netw. Sci. Eng., vol. 7, pp. 1508-1520, 2020.




\bibitem{Liberzon2005BB}
D.~Liberzon,  \emph{Switched systems}, Handbook of networked and embedded control systems, Birkh\"{a}user Boston, 2005.

 \bibitem{Cunha2021ECC}
R. Cunha, L. Zino, and M.  Cao,  \emph{On imitation dynamics in population games with Markov switching}, European Control Conference (ECC), pp. 722-727, 2021.




 \bibitem{Hilbe2018Nature}
C. Hilbe, \v{S}. \v{S}imsa,  K. Chatterjee,  and M. A. Nowak,  \emph{Evolution of cooperation in stochastic games}, Nature, vol. 559(7713), pp.  246-249, 2018.




\bibitem{Su19PNAS}
Q. Su, A. McAvoy, L. Wang, and  M. A. Nowak,  \emph{Evolutionary dynamics with game transitions}, Proc. Natl. Acad. Sci. U.S.A., vol. 116, pp. 25398-25404, 2019.

\bibitem{ShuPRSA2022}
L. Shu, and F.  Fu, \emph{Eco-Evolutionary Dynamics of Bimatrix Games}, Proc. Roy. Soc. A, vol. 478, pp. 20220567, 2022.


\bibitem{LiNC2020}
A.~Li,  L.~Zhou, Q.~Su, S.~P.~Cornelius, Y.-Y. Liu,  L. Wang, and S. A.  Levin,  \emph{Evolution of cooperation on temporal networks}, Nat. Commun., vol. 11, pp. 1-9, 2020.



\bibitem{Barab12NS}
A.~L.~Barab\'{a}si,  \emph{Network science}, New York, NY: Cambridge University Press, 2016.



































\bibitem{Khalil2002}
H.K.~Khalil,  \emph{Nonlinear Systems}, Printice-Hall, 2002.

\bibitem{Jones1995DS}
C.~K.~R.~T.~Jones,  \emph{Geometric singular perturbation theory}, Dynamical systems, Springer, Berlin, Heidelberg, 1995: 44-118.





\bibitem{Doebeli05EL}
M.~Doebeli, and C.~Hauert,  \emph{Models of cooperation based on the Prisoner's Dilemma and the Snowdrift game}, Ecol. Lett., vol. 8, pp. 748-766, 2005.

\bibitem{Skyrms04}
B.~Skyrms,  \emph{The stag hunt and the evolution of social structure}, Cambridge, UK: Cambridge University Press, 2004.

\bibitem{Huo17RSOS}
X.~Huo, and F.~Fu,  \emph{Risk-aware multi-armed bandit problem with application to portfolio selection}, Royal Soc. Open Sci., vol. 4(11), pp. 171377, 2017.

\bibitem{Hu21Games}
K.~Hu, and F.~Fu,  \emph{Evolutionary dynamics of gig economy labor strategies under technology, policy and market influence}, Games, vol. 12(2), pp. 49, 2021.

\bibitem{Wang20PRSA}
X.~Wang, Z.~Zheng, and F.~Fu,  \emph{Steering eco-evolutionary game dynamics with manifold control}, Proc. R. Soc. A, vol. 476(2233), pp. 20190643, 2020.

\bibitem{Watts98nature}
D.~J.~Watts, and S.~H.~Strogatz,  \emph{Collective dynamics of `small-world' networks}, Nature, vol. 393, pp. 440-442, 1998.

\end{thebibliography}
\end{document}